\documentclass[review,onefignum,onetabnum,final]{siamart220329}

\usepackage{lipsum}
\usepackage{amsmath}
\usepackage{amsfonts}
\usepackage{graphicx, subcaption}
\usepackage{epstopdf}
\usepackage{algorithmic}
\usepackage{amssymb,thmtools,mathtools,todonotes,stmaryrd}
\usepackage{xr}
\usepackage{hyperref}
\usepackage{cleveref}
\usepackage{autonum}
\usetikzlibrary{decorations.pathreplacing,calligraphy}
\ifpdf
  \DeclareGraphicsExtensions{.eps,.pdf,.png,.jpg}
\else
  \DeclareGraphicsExtensions{.eps}
\fi

\newsiamremark{remark}{Remark}
\newsiamremark{hypothesis}{Hypothesis}
\crefname{hypothesis}{Hypothesis}{Hypotheses}
\newsiamthm{claim}{Claim}

\headers{Computational UC with finite dimensional Neumann trace}{Erik Burman, Lauri Oksanen and Ziyao Zhao}

\title{Computational unique continuation with finite dimensional Neumann trace
\thanks{Submitted to the editors 23 Feb 2024.
\funding{E.B. was supported by the EPSRC grants EP/T033126/1 and EP/V050400/1. For the purpose of open access, the author has applied a Creative Commons Attribution (CC BY) licence to any Author Accepted Manuscript version arising.
L.O. and Z.Z. were supported by the European Research Council of the European Union, grant 101086697 (LoCal), and the Reseach Council of Finland, grants 347715, 353096 and 359182. Z.Z was supported by the Finnish Ministry of Education and Culture’s Pilot for Doctoral Programmes (Pilot project Mathematics of Sensing, Imaging and Modelling). Views and opinions expressed are those of the authors only and do not necessarily reflect those of the European Union or the other funding organizations.}
}}

\author{Erik Burman\thanks{Department of Mathematics, University College London, 807b Gower Street, London, WC1E 6BT, United Kingdom (\email{e.burman@ucl.ac.uk}).}
\and Lauri Oksanen\thanks{Department of Mathematics and Statistics, University of Helsinki, P.O 68, 00014, Helsinki, Finland (\email{lauri.oksanen@helsinki.fi}, \email{ziyao.zhao@helsinki.fi}).}
\and Ziyao Zhao\footnotemark[3]}

\usepackage{amsopn}

\def\p{\partial}
\def\R{\mathbb R}
\def\O{\varTheta}
\def\x{\mathbf x}
\def\T{\mathcal T}
\def\ta{\tilde{a}}
\def\tu{\tilde{u}}
\def\tz{\tilde{z}}
\def\tg{\tilde{g}}
\def\bb{\boldsymbol{b}}

\DeclareMathOperator{\supp}{supp}
\DeclarePairedDelimiter{\norm}{\lVert}{\rVert}
\DeclarePairedDelimiter{\abs}{\lvert}{\rvert}
\DeclarePairedDelimiter{\trinorm}{\interleave}{\interleave}

\makeatletter
\newcommand*{\addFileDependency}[1]{
  \typeout{(#1)}
  \@addtofilelist{#1}
  \IfFileExists{#1}{}{\typeout{No file #1.}}
}
\makeatother

\begin{document}
\nolinenumbers
\maketitle

\begin{abstract}
  We consider finite element approximations of unique continuation problems subject to elliptic equations in the case where the normal derivative of the exact solution is known to reside in some finite dimensional space. To give quantitative error estimates we prove Lipschitz stability of the unique continuation problem in the global $H^1$-norm. This stability is then leveraged to derive optimal a posteriori and a priori error estimates for a primal-dual stabilized finite method.
\end{abstract}


\section{Introduction}
In this work we are interested in the approximation of a unique continuation problem subject to the Poisson equation. This means that the we look for the solution to the equation,
\begin{equation}\label{eq:Poisson}
-\Delta u = f
\end{equation}
in $\Omega \subset \mathbb{R}^d$, and
for some $f \in L^2(\Omega)$, when the boundary condition is unavailable on the boundary, or part of the boundary. In its stead some measured data is available. Typically both Dirichlet and Neumann data are known on some part of the boundary (the elliptic Cauchy problem) or some measurement in the bulk. Both these situations can be handled using the arguments below, but for conciseness we will here concentrate on the second case. Therefore we assume that for some $\omega \subset \Omega$ there is $q:\omega \to \mathbb{R}$ such that $q$ is the restriction to $\omega$ of a solution to \eqref{eq:Poisson} and that we know $q$ up to a quantifiable perturbation $\delta q$. The objective is then to reconstruct $u$ using the equation \eqref{eq:Poisson} and the a priori knowledge $u\vert_\omega = q$.

Unique continuation is an important model problem for many applications in control, data assimilation or inverse problems. It is an ill-posed problem, so the assumption that the data $q$ is associated to a solution is crucial for the solvability of the problem. It is however well known that if this is the case a unique solution exists and satisfies a conditional stability estimate \cite{Giovanni_2009}. If $B \subset \subset \Omega$, that is the set $B$ does not intersect the boundary of $\Omega$ then there holds for all $u \in H^1(\Omega)$,
\[
\|u\|_{L^2(B)} \leq C (\|u\|_{L^2(\Omega)} + \|\Delta u\|_{H^{-1}(\Omega)})^{(1-\alpha)}(\|u\|_{L^2(\omega)} + \|\Delta u\|_{H^{-1}(\Omega)})^\alpha 
\]
where $\alpha \in (0,1)$. The coefficient $\alpha$ depends on the geometry of the domains $\Omega$, $\omega$ and $B$. As $\mbox{dist}(\partial B, \partial \Omega) \rightarrow 0$, $\alpha \rightarrow 0$. In case $B$ and $\Omega$ coincide the stability degenerates to logarithmic
\begin{equation}\label{eq:glob_stab}
    \|u\|_{L^2(\Omega)} \leq C \|u\|_{H^1(\Omega)} \log\left(\frac{\|u\|_{H^1(\Omega)}}{(\|u\|_{L^2(\omega)} + \|\Delta u\|_{H^{-1}(\Omega)})}\right)^{-\beta}
\end{equation}
with $\beta \in (0,1)$. In a series of works \cite{Bu16,BLO,BDE21} various finite element methods (FEM) have been designed and shown to satisfy bounds of the type
\begin{equation}\label{eq:loc_bound}
\|u - u_h\|_{L^2(B)} \leq C h^{\alpha k} (|u|_{H^{k+1}(\Omega)} + h^{-k} \|\delta q \|_{L^2(\omega)}),
\end{equation}
or if $B = \Omega$ and $\delta q =0$
\begin{equation}\label{eq:glob_bound}
\|u - u_h\|_{L^2(\Omega)} \leq C |u|_{H^{k+1}(\Omega)}  \log\left(C h^{-k} \right)^{-\beta}.
\end{equation}
In the recent contribution \cite{burman2023optimal}, the error bound on the form \eqref{eq:loc_bound} was shown to be optimal. It can not be improved for general  solutions and perturbations, regardless of the method. For moderately perturbed data and favourable subdomains $\omega$ and $B$ this leads to sufficient accuracy, in some cases comparable to that of a well-posed problem. On the other hand, if the solution, or its normal derivative, is required on the boundary of the domain the above estimates are very poor. Indeed, there seems to be no results on how to approximate boundary traces accurately in unique continuation problems. In view of the result in \cite{burman2023optimal}, the only way to improve on the bounds is to have additional a priori knowledge. In the work \cite{burman2023finite} it was shown that if the Dirichlet boundary trace is close to some known finite dimensional space then a FEM can be designed so that \eqref{eq:loc_bound} holds with $\alpha = 1$, with an additional perturbation term measuring the distance of the true solution to the finite dimensional space. The assumption that the trace is close to a finite dimensional space holds in a variety of situations, for instance whenever it is a smooth perturbation of a constant, in optimal control with finite dimensional boundary control, or in engineering applications where strong modelling a priori knowledge is at hand, for example classes of admissible boundary profiles.

In the present work we consider the extension of these results to the case when the Neumann condition is in a finite dimensional space. The main result is Corollary \ref{optimal_rate}, which gives optimal convergence rate for the finite element solution. Contrary to \cite{burman2023finite} we prove the stability underpinning the numerical analysis without resorting to the global stability \eqref{eq:glob_stab} (see Section \ref{sec:Lip_stab}). This makes the present analysis self contained. Although the proposed finite element method introduced in Section \ref{sec:FEM} is similar to that of \cite{burman2023finite}, the analysis differs in the Neumann case. Indeed the poorer regularity of the trace variable and the different functional analytical framework lead to some difficulties in the numerical analysis, that are handled in Section \ref{sec:FEM_error}, resulting in optimal a posterori and a priori error estimates. 
\subsection{Relation to previous work}
Early work on computational unique continuation (UC) focused on rewriting the problem as a boundary integral \cite{CM79, IYH91}, while the earliest finite element reference appears to be \cite{FM86}. The dominating regularization techniques are Tikhonov regularization \cite{TA77} and quasi reversibility \cite{LL69}. The literature on computational methods for the discretization of the regularized problem is very rich, see \cite{IJ15} and references therein. For references relevant for the present context we refer to \cite{RHD99,Bou05, DHH13,BR18, BC20,laurent2022}. Iteration techniques using boundary integral formulations have been proposed in \cite{Joh04} and for methods using tools from optimal control we refer to \cite{KK95}.

The first weakly consistent methods with regularisation on the discrete level were introduced in \cite{Bu13}, with the first analysis for ill-posed problems in \cite{Bu14} and then developed further in \cite{Bu16,BLO, BO18, BNO19, BDE21,belgacem2022}. This approach is related to previous work on finite element methods for indefinite problems based on least squares minimization in $H^{-1}$ \cite{BLP97,BLP98}. More recent results using least squares minimization in dual norm for ill-posed problems can be found in \cite{nicolae2015,CIY22, dahmen2022squares}. 
Quantitative a priori error estimates have been derived in a number of situations with careful analysis of the effect of the physical parameters of the problem on the constants of the error estimates \cite{BNO19,BNO20,BNO22}. This has lead to a deeper understanding of the computational difficulty of recovering quantities via UC in different parameter regimes.

That Lipschitz stability can be recovered for finite dimensional target quantities has been known for some time in the inverse problem community, see for example \cite{AV05, Bour13}. Nevertheless, it appears that the first time this property has been exploited in a computational method, leading to optimal error estimates in \cite{burman2023finite}.

\section{Problem setting} 

Let $\mathcal{V}_N$ be a subspace of $L^2(\p\Omega)$ that satisfies
\begin{enumerate}
  \item For all $g\in \mathcal{V}_N$, there holds $\int_{\p\Omega} g\ dx=0$,
  \item  $\dim(\mathcal{V}_N)=N<\infty$.
\end{enumerate}
We consider the following problems:
    \begin{equation}
    \label{question}
        \begin{cases}
          -\Delta u=f\ \text{in }\Omega,\\
          u=q\ \text{in }\omega,\\
          \partial_{\nu}u|_{\partial \Omega}\in\mathcal{V}_N+\beta,
        \end{cases}
    \end{equation}
      where $\Omega\in\mathbb{R}^d$ is an open, bounded polygonal domain, $\omega\subset\Omega$ is open, and nonempty. $f\in L^2(\Omega)$ and $\beta$ is a constant satisfying $\beta \left|\p\Omega\right| = \int_\Omega f\ dx$. We denote $P$ a projection operator on $\mathcal{V}_N$, with $Q=1-P$.
      
      \section{Lipschitz stability}\label{sec:Lip_stab}
      Firstly, we define a continuous bilinear functional $l(\cdot,\cdot)$ on $H^1(\Omega)\times H^1(\Omega)$
      \begin{equation}
        l(u,v)=(\nabla u,\nabla v)_{L^2(\Omega)},
      \end{equation}
      and for each $u\in H^1(\Omega)$, a continuous linear functional $L_u$ on $H^1(\Omega)$ 
        \begin{equation}
          L_u(v):=l(u,v).
        \end{equation}
      Next, we introduce the space
      \begin{equation}
        H^1_\omega(\Omega):=\left\{u\in H^1(\Omega)\mid \int_\omega u\ dx=0 \right\}.
      \end{equation}
      Then, we have the following lemma.
      \begin{lemma}
        \label{lemma_1}
        For every $u\in H^1_\omega(\Omega)$, there holds
        \begin{equation}
          \norm{u}_{H^1(\Omega)}\lesssim \norm{L_u}_{(H^1(\Omega))^*}.
        \end{equation}
      \end{lemma}
      \begin{proof}
      \quad\   We start from the inequality
        \begin{equation}
          \norm{u}_{H^1(\Omega)}\lesssim \norm{u}_{L^2(\Omega)}+\norm{\nabla u}_{L^2(\Omega)}+\norm{u}_{L^2(\omega)},
        \end{equation}
        define a continuous linear operator
        \begin{align}
          \label{define_A}
          A:\ H^1(\Omega)&\to \left[ L^2(\Omega)\right]^n\times L^2(\omega),\\
          Au&=(\nabla u,u|_\omega),
        \end{align}
        and denote the natural imbedding from $H^1(\Omega)$ to $L^2(\Omega)$ by $K$.
        To simplify the notation, we denote
        \begin{align}
          X:&= \left[ L^2(\Omega)\right]^n\times L^2(\omega).
        \end{align}
        Then we have
        \begin{equation}
          \Vert u\Vert_{H^1(\Omega)}\lesssim \Vert Au\Vert_{X}+\Vert Ku\Vert_{L^2(\Omega)}.
        \end{equation}
        Notice that $K$ is compact and $A$ is an injection. Indeed, suppose $Au_0=0$ for $u_0\in H^1(\Omega)$, that is, $\nabla u_0=0$ and $u_0|_\omega=0$. Since $\nabla u_0=0$ implies $u_0$ is a constant, $u_0=0$ in $\Omega$ follows immediately from $u_0|_\omega=0$.\par
        Therefore, by compactness-uniqueness \cite[Lemma 9]{Chervova},
        \begin{equation}
        \label{compactness_uniqueness}
          \norm{u}_{H^1(\Omega)}\lesssim \norm{Au}_{X}=\norm{\nabla u}_{L^2(\Omega)}+\norm{u}_{L^2(\omega)}.
        \end{equation}
        According to the Friedrich's inequality \cite[Lemma 4.3.14]{Brenner}, the condition $\int_\omega u\ d\x=0$ implies that
        \begin{equation}
          \norm{u}_{L^2(\omega)}\lesssim \norm{\nabla u}_{L^2(\omega)}\leq \norm{\nabla u}_{L^2(\Omega)}.
        \end{equation}
        Hence 
        \begin{equation}
          \label{Poincare_inequality}
          \norm{u}_{H^1(\Omega)} \lesssim \norm{\nabla u}_{L^2(\Omega)}.
        \end{equation}
        For arbitrary $\epsilon>0$, there holds
        \begin{equation}
          \label{ineq_1}
          \norm{\nabla u}^2_{L^2(\Omega)}=L_u(u)\leq \norm{L_u}_{(H^{1}(\Omega))^*}\norm{u}_{H^1(\Omega)}\leq \epsilon^{-1}\norm{L_u}_{(H^{1}(\Omega))^*}^2+\frac{\epsilon}{4}\norm{u}_{H^1(\Omega)}^2.
        \end{equation}
        Combining \eqref{Poincare_inequality} and \eqref{ineq_1}, and choosing $\epsilon$ small enough, we have
        \begin{equation}
          \norm{u}_{H^1(\Omega)}\lesssim \norm{L_u}_{(H^{1}(\Omega))^*}.
        \end{equation}
      \end{proof}  
      \begin{lemma}
        \label{existence}
        Suppose $F\in (H^1(\Omega))^*$ satisfies
        \begin{equation}
          F(\mathbf{c})=0,
        \end{equation}
        for all constant functions $\mathbf{c}$ on $\Omega$. Then for the variational problem
        \begin{equation}
          \label{eq:l=F}
          l(u,v)=F(v),\ \forall v\in H^1(\Omega)    
        \end{equation}
        there exists a unique solution $u\in H^1_\omega(\Omega)$.
      \end{lemma}
      \begin{proof}
        Since for any constant function $\mathbf{c}$ on $\Omega$, $F(\mathbf{c})=0$ and $l(u,\mathbf{c})=0$, problem \eqref{eq:l=F} amounts to finding $u\in H^1(\Omega)$ such that
        \begin{equation}
          \label{eq:l=F_1}
          l(u,v)=F(v),\ \forall v\in H^1_{\omega}(\Omega).
        \end{equation}
        The coercivity of $l$ on $H^1_{\omega}(\Omega)$ follows from \eqref{Poincare_inequality}. According to Lax-Milgram theorem \cite[Corollary 5.8]{Brezis}, there exists a unique $u\in H^1_{\omega}(\Omega)$ solves \eqref{eq:l=F}.
      \end{proof}
      For each $u\in H^1(\Omega)$ with $\p_\nu u|_{\p\Omega}\in L^2(\p\Omega)$, we define a linear functional $L^q_u$ on $H^1(\Omega)$
      \begin{equation}
        L^q_u(v):=L_u(v)-(P\p_\nu u,v)_{L^2(\p\Omega)},\ v\in H^1(\Omega).
      \end{equation}
      By trace inequality, we can see that $L^q_u\in (H^1(\Omega))^*$.
      \begin{theorem}
        \label{lipschitz_stability}
        For $u\in H^1(\Omega)$ with $\p_\nu u|_{\p\Omega}\in L^2(\p\Omega)$, there holds
      \begin{equation}
      \label{eq:Lipschitz_stability}
        \norm{P\p_\nu u}_{L^2(\p\Omega)}+\Vert u\Vert_{H^1(\Omega)}\lesssim \Vert u\Vert_{L^2(\omega)}+\norm{L^q_u}_{(H^1(\Omega))^*}.
      \end{equation} 
      \end{theorem}
      
      \begin{proof}
        First we assume that $u\in H^1_\omega(\Omega)$.\par
        Write $u=v+w$, where $v,\ w\in H^1_\omega(\Omega)$ and satisfy
        \begin{equation}
          \label{equation_v}
              l(v,\varphi)=L_u^q(\varphi),\ \forall\varphi\in H^1(\Omega),
            \end{equation}
        and
        \begin{equation}
          \label{equation_w}
            l(w,\varphi)=(P\p_\nu u,\varphi)_{L^2(\p\Omega)},\ \forall \varphi\in H^1(\Omega),
        \end{equation}
        respectively. Since $L^q_u(\mathbf{c})=0$ and $(P\p_\nu u, \mathbf{c})_{L^2(\p\Omega)}=0$, \eqref{equation_v} and \eqref{equation_w} are both solvable by Lemma \ref{existence}. We define an operator
        \begin{align}
          A&:\mathcal{V}_N\to L^2(\omega)\\
          A(g)&:= w_g|_\omega,\ g\in \mathcal{V}_N,
        \end{align}
        where $w_g\in H^1_\omega(\Omega)$ satisfies 
        \begin{equation}
          l(w_g,\varphi)=(g,\varphi)_{L^2(\p\Omega)},\ \forall\varphi\in H^1(\Omega).
        \end{equation}
        Notice that $\norm{\Delta w_g}_{H^{-1}(\Omega)}=0$. Then by \eqref{eq:glob_stab}, $A$ is injective.\par
        Since $A(\mathcal{V}_N)$ is a finite dimensional subspace of $L^2(\omega)$, there exist a norm on $A(\mathcal{V}_N)$ such that $A$ is an isometry. As all norms are equivalent in the finite dimensional range of $A$, there holds
        \begin{equation}
            \norm{g}_{L^2(\p\Omega)}\lesssim \norm{A g}_{L^2(\omega)}.
        \end{equation}
        Notice that $A(P\p_\nu u)=w|_\omega$. We conclude that
        \begin{equation}
          \norm{P\p_\nu u}_{L^2(\p\Omega)}\lesssim \norm{w}_{L^2(\omega)}.
        \end{equation}
        Applying the Lemma \ref{lemma_1} to $v$ and $w$, there holds
        \begin{align}
          \label{ineq_2}
          \norm{u}_{H^1(\Omega)}&\leq \norm{v}_{H^1(\Omega)}+\norm{w}_{H^1(\Omega)}\lesssim \norm{v}_{H^1(\Omega)}+\norm{P\p_\nu u}_{L^2(\p\Omega)}\\
          &\lesssim \norm{v}_{H^1(\Omega)}+\norm{w}_{L^2(\omega)}\lesssim \norm{v}_{H^1(\Omega)}+\norm{u}_{L^2(\omega)}\\
          &=\norm{L_u^q}_{(H^1(\Omega))^*}+\norm{u}_{L^2(\omega)}.
        \end{align}
        The above argument also gives
        \begin{equation}
            \norm{P\p_\nu u}_{L^2(\p\Omega)}\lesssim \norm{L_u^q}_{(H^1(\Omega))^*}+\norm{u}_{L^2(\omega)}.
        \end{equation}
        \par
        For $u \not\in H^1_\omega(\Omega)$, there exist a $\tilde{u}\in H^1_\omega(\Omega)$ and a constant $C\in\R$ such that $u=C+\tilde{u}$. Then
        \begin{align}
          \norm{u}_{H^1(\Omega)}^2&\leq2\norm{\tilde{u}}_{H^1(\Omega)}^2+2C^2|\Omega|,\\
          \norm{u}^2_{L^2(\omega)}&= \norm{\tilde{u}}^2_{L^2(\omega)}+C^2|\omega|.
        \end{align}
        Therefore, we have
        \begin{align}
          \norm{u}_{H^1(\Omega)}^2&\leq 2\norm{\tilde{u}}^2_{H^1(\Omega)}+2C^2|\Omega|\lesssim \norm{L_{\tilde{u}}^q}_{(H^1(\Omega))^*}^2+\norm{\tilde{u}}^2_{L^2(\omega)}+C^2|\omega|\\
          &= \norm{L_u^q}_{(H^1(\Omega))^*}^2+\norm{u}_{L^2(\omega)}^2,
        \end{align}
        and
        \begin{align}
            \norm{P\p_\nu u}_{L^2(\p\Omega)}&=\norm{P\p_\nu \tilde{u}}_{L^2(\p\Omega)}\lesssim \norm{\tilde{u}}_{L^2(\omega)}+\norm{L^q_{\tilde{u}}}_{(H^1(\Omega))^*}\\
            &\lesssim \norm{u}_{L^2(\omega)}+\norm{L^q_u}_{(H^1(\Omega))^*}.
        \end{align}
        Combining the above two inequalities yields \eqref{eq:Lipschitz_stability}.
      \end{proof}
      
      \section{Finite element method}\label{sec:FEM}
      Here we will introduce a finite element method for the approximation of \eqref{question}. Some results detailing the continuity, stability and consistency properties of the method will then be proven, preparing the terrain for the error analysis in the next section.
      
      Let $\T_h$ be a decomposition of $\Omega$ into shape regular simplices $K$ that form a simplicial complex and let $h=\max_{K\in\T_h}\text{diam}(K)$ be the global mesh parameter. The trace inequality with scaling \cite[Eq. 10.3.8]{Brenner} reads
      \begin{equation}
        \label{eq:trace_inequality}
        h^{1/2}\norm{u}_{L^2(\p K)}\lesssim \norm{u}_{L^2(K)}+\norm{h\nabla u}_{L^2(K)},\quad u\in H^1(K).
      \end{equation}
      On $\T_h$ we define the standard space of continuous finite element functions
      \begin{equation}
        V_h:=\{v\in H^1(\Omega)\mid \ v|_K\in \mathbb{P}_k \text{ for } K\in \T_h\}.
      \end{equation}
      Here $\mathbb{P}_k$ is the space of polynomial of degree at most $k\geq 1$ on $K$. For $m\geq 0$, we denote the broken semiclassical Sobolev seminorms and norms by
      \begin{equation}
        [u]_{H^m(\T_h)}^2=\sum_{K\in \T_h}\norm{(h D)^m u}_{L^2(K)}^2, \quad \norm{u}_{H^m(\T_h)}^2=\sum_{k=0}^m [u]^2_{H^k(\T_h)}.
      \end{equation} 
      The discrete inequality \cite[Lemma 1.138]{ern} in $\mathbb{P}_k$ implies that for all integers $m\geq l\geq 0$
      \begin{equation}
        \label{eq:inverse}
        [u]_{H^m(\T_h)}\lesssim [u]_{H^l(\T_h)},\quad u\in V_h.
      \end{equation}
      Let $F$ be a interior face between two simplices $K_1,K_2\in \mathcal{T}_h$ such that $K_1\cap K_2=F$, then the jump over $F$ is given by
      \begin{equation}
          \llbracket \nabla u \rrbracket_F = \nu_1\cdot \nabla u|_{K_1} + \nu_2\cdot \nabla u|_{K_2},
      \end{equation}
      where $\nu_1$ and $\nu_2$ are the outward normal of $K_1$ and $K_2$, respectively.
      
      Then we introduce the Lagrangian on $V_h \times V_h$
      \begin{align}
      \label{eq:def_lagrangian}
        \mathcal{L}(u,z)&:=\frac{h^2}{2}\norm{u-q}_{L^2(\omega)}^2+a(u,z)-h^2(f,z)_{L^2(\Omega)}-h^2(\beta,z)_{L^2(\p\Omega)}\\
        &+\frac{1}{2}B(u) - h(h Q \p_\nu u, h\beta)_{L^2(\p\Omega)}+\frac{1}{2}[h^2\Delta u+h^2 f]_{H^0(\T_h)}^2+\frac{1}{2}J(u)-\frac{1}{2}S^*(z),
      \end{align}
      where 
      \begin{align}
        \label{definitions}
        B(u)&=h\norm{h Q\p_\nu u}^2_{L^2(\p\Omega)},\\
        a(u,z)=(h\nabla u&,h\nabla z)_{L^2(\Omega)}-h\int_{\p\Omega}h(P\p_\nu u) z\ dx,\\
        J(u)&=\sum_{K\in \T_h}h\norm{\llbracket h\nabla u\rrbracket}_{L^2(\p K\setminus \p\Omega)}^2,\\
        S^*(z)&=h^2\norm{z}_{H^1(\Omega)}^2.
      \end{align}
      For 
      \begin{align}
        S(u)&=J(u)+[h^2\Delta u]^2_{H^0(\T_h)},
      \end{align}
      we have the following lemma.
      \begin{lemma}
        \label{a_c_estimate}
        For $u\in H^1(\Omega)$ with $\Delta (u|_K)\in L^2(K)$ for all $K\in \T_h$, $\p_\nu u|_{\p\Omega}\in L^2(\p\Omega)$ and $z\in H^1(\Omega)$, there holds
        \begin{align}
          a(u,z)\lesssim (S(u)^{1/2}+B(u)^{1/2})\norm{z}_{H^1(\T_h)}.
        \end{align}
      \end{lemma}
      \begin{proof}
        An integration by parts reads
        \begin{align}
          &a(u,z)=(h\nabla u,h\nabla z)_{L^2(\Omega)}-h\int_{\p\Omega}h(P \p_\nu u)z\ dx\\
          &=-\sum_{K\in\T_h}\int_K h^2\Delta u z\ dx+\sum_{F\in\mathcal{F}_h}h\int_F\llbracket h\p_\nu u\rrbracket z\ dx+h\int_{\p\Omega}(Qh\p_\nu u)z\ dx,
        \end{align}
        where $\mathcal{F}_h$ is the set of elements faces in the interior of $\Omega$.\par
        Applying the trace inequality with scaling \eqref{eq:trace_inequality} on each face in $\mathcal{F}_h$ and $\p\Omega$, we can obtain
        \begin{align}
          &\sum_{F\in\mathcal{F}_h}h\int_F \llbracket h\p_\nu u\rrbracket z\ dx\leq \sum_{K\in\T_h}h^{1/2}\norm{\llbracket h\p_\nu u\rrbracket}_{L^2(\p K\setminus \p\Omega)}h^{1/2} \norm{z}_{L^2(\p K\setminus \p\Omega)}\\
          &\lesssim \sum_{K\in\T_h}h^{1/2}\norm{\llbracket h\p_\nu u\rrbracket}_{L^2(\p K\setminus \p\Omega)}\left(\norm{z}_{L^2(K)}+\norm{h \nabla z}_{L^2(K)}\right)\\
          &\leq \left(\sum_{K\in\T_h} h \norm{\llbracket h\p_\nu u\rrbracket}_{L^2(\p K\setminus \p\Omega)}^2\right)^{\frac{1}{2}}\left(\sum_{K\in\T_h}\norm{z}_{L^2(K)}^2+\norm{h\nabla z}_{L^2(K)}^2 \right)^{\frac{1}{2}}\\
          &\leq J(u)^{1/2}\norm{z}_{H^1(\T_h)},
        \end{align}
        and
        \begin{align}
          h\int_{\p\Omega}(Qh\p_\nu u)z\ dx\leq h^{1/2}\norm{Qh\p_\nu u}_{L^2(\p\Omega)}h^{1/2}\norm{z}_{L^2(\p\Omega)}\lesssim B(u)^{1/2}\norm{z}_{H^1(\T_h)}.
        \end{align}
        Finally, notice that
        \begin{align}
          \sum_{K\in\T_h}\int_K h^2(\Delta u)z\ dx&\leq \sum_{K\in\T_h}[h^2\Delta u]_{L^2(K)}\norm{z}_{L^2(K)}\\
          &\leq [h^2\Delta u]_{H^0(\T_h)} [z]_{H^0(\T_h)}\leq [h^2\Delta u]_{H^0(\T_h)}\norm{z}_{H^1(\T_h)}.
        \end{align}
        Thus we conclude that
        \begin{align}
          a(u,z)&\lesssim \left([h^2\Delta u]_{H^0(\T_h)}+J(u)^{1/2}+B(u)^{1/2}\right)\norm{z}_{H^1(\T_h)}\\
          &\lesssim (S(u)^{1/2}+B(u)^{1/2})\norm{z}_{H^1(\T_h)},
        \end{align}
      \end{proof}
      Finding saddle points of the Lagrangian $\mathcal{L}$ amounts to finding $(u,z)\in V_h \times V_h$ such that for all $(v,w)\in V_h \times V_h$, there holds
      \begin{equation}
        \begin{cases}
          a(u,w)-s^*(z,w)=h^2(f,w)_{L^2(\Omega)}+h^2(\beta,w)_{L^2(\p\Omega)},\\
          h^2(u,v)_{L^2(\omega)}+b(u,v)+a(v,z)+s(u,v)=h^2(q,v)_{L^2(\omega)}+h^3(Q \p_\nu v, \beta)_{L^2(\p\Omega)}\\+(h^2f,h^2\Delta v)_{H^0(\T_h)}.
        \end{cases}
    \end{equation}
      Here $s^*,\ b$ and $s$ are the bilinear forms corresponding to $S^*$, $B$ and $S$, respectively.
      Solving the above system is equivalent to finding $(u,z)\in V_h\times V_h$ such that for all $(v,w)\in V_h\times V_h$, there holds
      \begin{align}
        \label{equation_g}
        g(u,z,v,w)&=h^2(f,w)_{L^2(\Omega)}+h^2(\beta,w)_{L^2(\p\Omega)}+h^2(q,v)_{L^2(\omega)}\\
        &+h^3(Q \p_\nu v, \beta)_{L^2(\p\Omega)}+(h^2f,h^2\Delta v)_{H^0(\T_h)},
      \end{align}
      where the bilinear form $g$ is defined by
      \begin{align}
        g(u,z,v,w):=h^2(u,v)_{L^2(\omega)}+b(u,v)+a(v,z)+s(u,v)+a(u,w)-s^*(z,w).
      \end{align}
      
      \begin{remark}
      Recall that $f \in L^2(\Omega)$. 
          Hence, if $u\in H^1(\Omega)$ solves \eqref{question} then $u\in H^2_{\text{loc}}(\Omega)$. In particular, for all $F \in \mathcal F_h$ and all compact sets $K$ in the interior of $\Omega$ there holds $\llbracket\nabla u \rrbracket = 0$ on $F \cap K$. 
          As $K$ is arbitrary, the same holds on the whole set $F$, and 
          $J(u)=0$. Therefore, $(u,0)$ solves \eqref{equation_g} for all $(v,w)\in V_h\times V_h$ and the system \eqref{equation_g} is consistent.
      \end{remark}

      Preparing the terrain for the error analysis of the next section, we introduce the norm 
      \begin{equation}
        \trinorm{u,z}^2=B(u)+h^2 \norm{u}_{L^2(\omega)}^2+S(u)+S^*(z).
      \end{equation}
      According to Lemma \ref{a_c_estimate}, we have
      \begin{align}
        \norm{L^q_u}_{(H^1(\Omega))^*}&=\sup_{v\in H^1(\Omega)}\frac{L^q_u(v)}{\norm{v}_{H^1(\Omega)}}=\sup_{v\in H^1(\Omega)}\frac{a(u,v)}{h^2\norm{v}_{H^1(\Omega)}}\\
        &\lesssim \sup_{v\in H^1(\Omega)}\frac{(S(u)^{1/2}+B(u)^{1/2})\norm{v}_{H^1(\T_h)}}{h^2\norm{v}_{H^1(\Omega)}}\lesssim \frac{1}{h^2}\trinorm{u,0}.
      \end{align}
      By Theorem \ref{lipschitz_stability}, we have 
      \begin{equation}
          \norm{u}_{H^1(\Omega)}\lesssim \frac{1}{h^2}\trinorm{u,0}.
      \end{equation}
      Notice that $\trinorm{u,z}=\trinorm{u,0}+\trinorm{0,z}$, then $\trinorm{u,z}=0$ implies that $u=0$ and $z=0$. Hence $\trinorm{\cdot,\cdot}$ is indeed a norm.
       \begin{lemma}
        \label{lemma_3}
        \begin{equation}
          \trinorm{u,z}\lesssim \sup_{(v,w)\in V_h\times V_h}\frac{g(u,z,v,w)}{\trinorm{v,w}},\quad (u,z)\in V_h\times V_h. 
        \end{equation}
      \end{lemma}
      \begin{proof}
        Follows immediately from
        \begin{equation}
          \trinorm{u,z}^2=g(u,z,u,-z).
        \end{equation}
      \end{proof}
      Due to Lemma \ref{lemma_3}, the linear system \eqref{equation_g} admits a unique solution for any $f\in L^2(\Omega)$ and $q\in L^2(\omega)$. Denoting the unique solution by $(u_h,z_h)$ and letting $u$ solve \eqref{question}, we have the Galerkin orthogonality
      \begin{equation}
        \label{Galerkin_orthogonality}
        g(u_h-u,z_h,v,w)=0,\ (v,w)\in V_h\times V_h.
      \end{equation}
 \section{Necessity of regularization}\label{sec:necessity}

One may wonder if the regularization introduced in Section \ref{sec:FEM} is really needed. 
To illustrate necessity of a regularization of some form, let us consider a straightforward data fitting approach.  
To simplify the discussion, we suppose that $f=0$ and therefore $\beta=0$ in \eqref{question}. Let $\{\psi_i\}_{i=1}^N$ be a basis of $\mathcal{V}_N$ and $w^i_h\in V_h\cap H^1_\omega(\Omega)$ be the solution of
    \begin{equation}
\int_\Omega \nabla w^i_h\cdot \nabla v\, dx = \int_{\p\Omega} \psi_i v\, dx,\ \forall v\in V_h.
    \end{equation}
Let $W_h$ be the space spanned by $\{w^i_h\}_{i=1}^N\oplus \{1\}$ where $1$ is the constant function that is identical to $1$ in $\Omega$. 
It feels natural to try to approximate the solution of \eqref{question} 
via 
    \begin{equation}
\label{eq:simple_method}
u_h=\underset{w_h\in W_h}{\text{arg min}}\norm{w_h-q}_{L^2(\omega)}.
    \end{equation}
However, such $u_h$ may not be unique. Indeed, if there exists $w_h\in W_h$ that satisfies $w_h|_{\omega}=0$, then \eqref{eq:simple_method} does not define $u_h$ uniquely. 
In fact, this may happen even when $N=1$ if $h > 0$ is not small enough, as shown by Proposition \ref{prop:necessity} below.
In other words, the mesh size $h > 0$ can not be chosen independently of $\mathcal{V}$. 

Presumably there is $h_0(\mathcal{V},\Omega,\omega) > 0$ such that all $w_h \in W_h$ have the unique continuation property for 
\begin{equation}
  0\leq h\leq h_0(\mathcal{V},\Omega,\omega),
\end{equation}        
but we have opted not to try to find such a selection rule $h_0$. We expect that $h_0$ would depend on the implicit constant in \eqref{eq:Lipschitz_stability}. On the other hand, our method introduced in Section \ref{sec:FEM} does not impose any constraints on the mesh size, as the system
\eqref{equation_g} has a unique solution for any $h > 0$.

\begin{figure}[h!tbp]
    \centering
    \begin{tikzpicture}[scale=5]
      \def\h{0.2} 
      \def\n{5}   
      
      \foreach \i in {0,1,2,3,4,5} {
          \foreach \j in {0,1,2,3,4,5} {
              \node[fill=black,circle,scale=0.3] at ({\i*\h},{\j*\h}) {};
          }
      }
      
      \foreach \i in {0,1,2,3,4} {
          \foreach \j in {0,1,2,3,4} {
              \draw[thick] ({\i*\h},{\j*\h}) -- ({(\i+1)*\h},{\j*\h});
              \draw[thick] ({\i*\h},{\j*\h}) -- ({\i*\h},{(\j+1)*\h});
              \draw[thick] ({\i*\h},{(\j+1)*\h}) -- ({(\i+1)*\h},{\j*\h});
              \draw[thick] ({(\i+1)*\h},{\j*\h}) -- ({(\i+1)*\h},{(\j+1)*\h});
              \draw[thick] ({\i*\h},{(\j+1)*\h}) -- ({(\i+1)*\h},{(\j+1)*\h});
          }
      }
      
      \node[below left] at (0,0) {(0,0)};
      \node[below right] at (1,0) {(1,0)};
      \node[above left] at (0,1) {(0,1)};
      \node[above right] at (1,1) {(1,1)};
  \end{tikzpicture}
  \caption{A uniform triangular decomposition on $\Omega=[0,1]\times[0,1]$ with $n=6$.}  
  \label{fig:mesh}
  \end{figure}

To simplify the discussion, we suppose that $\Omega:=[0,1]\times[0,1]$ and the mesh size $h=\frac{\sqrt{2}}{n-1}$ where $n>3$ is an integer. Let $\mathcal{T}_h$ be a uniform triangular decomposition of $\Omega$ as exemplified in Figure \ref{fig:mesh}, and let the polynomial order of $V_h$ be one. Then $V_h$ is a $n^2$ dimensional space. The set of nodal basis functions $\{\phi_i\}_{i=1}^{n^2}$ forms a basis of $V_h$. We notice that $\{\phi_i\}_{i=1}^{n^2}$ consists of $4n-4$ boundary nodal basis functions and $(n-2)^2$ interior nodal basis functions.\par
  Let $\{\phi_i\}_{i=1}^{4n-4}$ be boundary nodal functions. We denote the space of continuous functions with zero mean value on $\p\Omega$ by $\Gamma$, that is 
  \begin{equation}
    \Gamma:=\{g\in C(\p\Omega) \mid \int_{\p\Omega} g\, dx=0\}.
  \end{equation}
  And we define the operator
  \begin{align}
    \Lambda: \Gamma\to \R^{4n-4}
  \end{align}
  as 
  \begin{equation}
    \Lambda (g)= \bb=(b_1,b_2,\cdots,b_{4n-4}),\ g\in \Gamma,
  \end{equation}
  where
  \begin{equation}
    b_i=\int_{\p\Omega}g\phi_i\, dx, 1\leq i\leq 4n-4.
  \end{equation}
  Write the vector $\boldsymbol{e}=(\underbrace{1,1,\cdots,1}_{4n-4})^T$. Consider the functional it induced on $\R^{4n-4}$
  \begin{equation}
    T(\boldsymbol{a})=\boldsymbol{e}^T \boldsymbol{a}=\sum_{i=1}^{4n-4}a_i,\ \boldsymbol{a}=(a_1,a_2,\cdots,a_{4n-4})\in\R^{4n-4}. 
  \end{equation}
  We have the following lemma.
  \begin{lemma}
    \label{lm:necessity}
    Let $R(\Lambda):=\{\Lambda g\mid g\in \Gamma\}$ be the range of $\Lambda$, then we have
    \begin{equation}
      T^{-1}(0)\subset R(\Lambda).
    \end{equation}
  \end{lemma}
  \begin{proof}
    Rearrange the index such that $\phi_i$ and $\phi_{i+1}$ are two adjacent boundary nodal basis functions for $1\leq i\leq 4n-5$. Without loss of generality, we assume that $\supp(\phi_i|_{\p\Omega})\cap \supp(\phi_{i+1}|_{\p\Omega})=[y_1,y_2]$ and
    \begin{equation}
      \phi_i|_{[y_1,y_2]}=1-\frac{y-y_1}{y_2-y_1},\ \phi_{i+1}|_{[y_1,y_2]}=\frac{y-y_1}{y_2-y_1}.
    \end{equation}
    Notice that $[y_1,y_2]\cap \supp(\phi_j|_{\p\Omega})=\emptyset$ if $j\neq i, i+1$.
    Then for a nonzero smooth function $g_i$ that is supported in $(y_1,y_2)$ and that is odd corresponding to the midpoint of the interval $[y_1,y_2]$, we have
    \begin{equation}
      \int_{\p\Omega} g_i\phi_i\, dx = -\int_{\p\Omega} g_i\phi_{i+1}\, dx=c_i\neq 0,
    \end{equation}
    and
    \begin{equation}
      \int_{\p\Omega} g_i \phi_j\, dx=0, \text{ if }j\neq i,\ i+1.
    \end{equation}
    It follows that
    \begin{equation}
      \Lambda\left(g_i/c_i\right) =  \bb_i=(\underbrace{0,\cdots,0}_{(i-1)\ zeros},1,-1,\underbrace{0,\cdots,0}_{(4n-i-5)\ zeros}),\ 1\leq i\leq 4n-5.
    \end{equation}
    Observe that $\{\bb_i\}_{i=1}^{4n-5}$ forms a basis of $T^{-1}(0)$, thus $T^{-1}(0)\subset R(\Lambda)$ follows immediately.
  \end{proof}
  \begin{proposition}
  \label{prop:necessity}
     Suppose $\Omega:=[0,1]\times[0,1]$. For any fixed integer $n>3$ and $h=\frac{\sqrt{2}}{n-1}$, we set $\omega:=[\frac{1}{n-1},1-\frac{1}{n-1}]\times [\frac{1}{n-1},1-\frac{1}{n-1}]$. Let $\mathcal{T}_h$ be the uniform triangular decomposition of $\Omega$ and the polynomial order of $V_h$ be one. Then there exists a function $g\in \Gamma$, such that the solution $u_h\in V_h\cap H^1_\omega(\Omega)$ of 
    \begin{equation}
    \label{eq:var}
    \int_\Omega \nabla u_h\cdot \nabla v\, dx = \int_{\p\Omega} g v\, dx,\ \forall v\in V_h.
  \end{equation}
    satisfies  $u_h\neq 0$ and $u_h|_{\omega}=0$.
  \end{proposition}
  
  \begin{proof}
    Let $\{\phi_i\}_{i=1}^{n^2}$ be nodal basis functions of $V_h$ and rearrange the index so that $\{\phi_i\}_{i=1}^{(n-2)^2}$ are the interior nodal basis functions while $\{\phi_i\}_{i=(n-2)^2+1}^{n^2}$ are the boundary nodal basis functions. Then we consider
  \begin{equation}
    \label{eq:linear_system}
    \boldsymbol{A} \boldsymbol{\alpha}=\boldsymbol{b},
  \end{equation}
  where $\boldsymbol{\alpha}=(\alpha_1,\cdots,\alpha_{n^2})^T$, the entries of $\boldsymbol{A}=(a_{ij})_{n^2\times n^2}$ are $a_{ij}=\int_\Omega \nabla \phi_i\cdot \nabla \phi_j\, dx$, and the entries of $\boldsymbol{b}=(b_j)_{n^2}$ are $b_j=\int_{\p\Omega} \psi \phi_j\, dx$ for some $\psi\in \Gamma$. Our first observation is that $b_j=0$ for $j\leq (n-2)^2$ since $\{\phi_j\}_{j=1}^{(n-2)^2}$ are interior nodal basis functions. Then we decompose \eqref{eq:linear_system} into the following form
  \begin{equation}
    \label{eq:linear_sys_1}
    \begin{pmatrix}
      A & B \\
      B^T & D
    \end{pmatrix}
    \begin{pmatrix}
      \boldsymbol{\alpha}_1\\
      \boldsymbol{\alpha}_2
    \end{pmatrix}
    =
    \begin{pmatrix}
      0 \\
      \boldsymbol{b}_2
    \end{pmatrix}
  \end{equation}
  Here $\boldsymbol{\alpha}_1\in \R^{(n-2)^2}$, $\boldsymbol{\alpha}_2,\ \bb_2\in \R^{4n-4}$, $A\in \R^{(n-2)^2\times (n-2)^2}$, $B\in \R^{(4n-4)\times (n-2)^2}$ and $D\in \R^{(4n-4)\times (4n-4)}$. For the interior nodal basis functions $\{\phi_j\}_{j=1}^{(n-2)^2}$, 
  \begin{equation}
    \int_\Omega \nabla\phi_j \nabla\phi_i\, dx\neq 0\ \text{for some } (n-2)^2+1\leq i \leq n^2
  \end{equation}
  only if $\phi_j$ is an element adjacent to the boundary. For uniform triangular decomposition, there are only $4n-12$ interior elements adjacent to boundary elements. That is, at most $4n-12$ rows in $B$ are nonzero.\par
  Let $\boldsymbol{\alpha}_1=0$. Since the dimension of $\boldsymbol{\alpha}_2$ is $4n-4$ and the rank of $B$ is $4n-12$, the linear system 
  \begin{equation}
    \begin{cases}
      B\boldsymbol{\alpha}_2=0\\
      \boldsymbol{e}^T D\boldsymbol{\alpha}_2=0 
    \end{cases}
  \end{equation}
  admits a 7-dimensional null space $X\subset \R^{4n-4}$. Select a nonzero $\boldsymbol{\alpha}_2\in X$. According to Lemma \ref{lm:necessity}, there exists a function $g\in \Gamma$ such that $\Lambda(g)=D\boldsymbol{\alpha}_2$. Recall that
  \begin{equation}
      \boldsymbol{\alpha}=(\alpha_1,\cdots,\alpha_{n^2})^T=\begin{pmatrix}
      0 \\
      \boldsymbol{\alpha}_2
    \end{pmatrix}.
  \end{equation}
  Set $u_h=\sum_{i=1}^{n^2} \alpha_i \phi_i$, then $u_h$ solves \eqref{eq:var}. We notice that $u_h|_{[h,1-h]\times [h,1-h]}=0$ since $\alpha_j=0,\ 1\leq j\leq (n-2)^2$, corresponding to the coefficients of interior nodal basis functions.
\end{proof}
      
\section{Error analysis}\label{sec:FEM_error}
       The objective of this section is to leverage the stability of Theorem \ref{lipschitz_stability} to derive optimal error estimates. We start with a technical lemma.
      \begin{lemma}
        \label{lemma_laplacian}
        Suppose that $u\in H^{1}(\Omega)$ and $(u_h,z_h)\in V_h\times V_h$ solve \eqref{question} and \eqref{equation_g}, respectively. Then there holds
        \begin{equation}
          \label{eq_lemma_laplacian}
          h^{-1}\norm{h^2 L^q_{u-u_h}}_{(H^1(\Omega))^*}\lesssim S(u-u_h)^{1/2}+B(u-u_h)^{1/2}+S^*(z_h)^{1/2}.
        \end{equation}
      \end{lemma}
      \begin{proof}
        We have
        \begin{align}
          \norm{h^2 L^q_{u-u_h}}_{(H^1(\Omega))^*}=\sup_{w\in H^1(\Omega)}\frac{h^2L^q_{u-u_h}(w)_{L^2(\Omega)}}{\norm{w}_{H^1(\Omega)}}=\sup_{w\in H^1(\Omega)}\frac{a(u-u_h,w)}{\norm{w}_{H^1(\Omega)}}.
        \end{align}
        Let $i_h: H^m(\Omega)\to V_h$ be an interpolator satisfying
        \begin{equation}
          \label{interpolation}
              \norm{u-i_h u}_{H^m(\T_h)}\lesssim [u]_{H^m(\T_h)},\ u\in H^m(\Omega).
        \end{equation}
        for all $m\geq 1$. The Scott-Zhang interpolator is a possible choice \cite{Scott-Zhang}.
        For any $w\in H^1(\Omega)$, \eqref{Galerkin_orthogonality} gives
        \begin{equation}
          g(u_h-u,z_h,0,i_h w)=a(u_h-u,i_h w)-s^*(z_h,i_h w)=0.
        \end{equation}
        Applying Lemma \ref{a_c_estimate}, we have
        \begin{equation}
          \begin{split}
          &a(u-u_h,w)=a(u-u_h,w-i_h w)-s^*(z_h,i_h w)\\
          &\lesssim (S(u-u_h)^{1/2}+B(u-u_h)^{1/2})\norm{w-i_h w}_{H^1(\T_h)}+S^*(z_h)^{1/2} S^*(i_h w)^{1/2}\\
          &\lesssim (S(u-u_h)^{1/2}+B(u-u_h)^{1/2})[w]_{H^1(\T_h)}+S^*(z_h)^{1/2}h(\norm{i_h w-w}_{H^1(\Omega)}+\norm{w}_{H^1(\Omega)})\\
          &\lesssim h\left(S(u-u_h)^{1/2}+B(u-u_h)^{1/2}+S^*(z_h)^{1/2}\right)\norm{w}_{H^1(\Omega)},
        \end{split}
        \end{equation}
        which implies \eqref{eq_lemma_laplacian}.
      \end{proof}
      Then we can prove the a posteriori error estimate 
      \begin{theorem}\label{thm:aposterori}
        Suppose that $u\in H^{1}(\Omega)$ and $(u_h,z_h)\in V_h$ solves \eqref{question} and \eqref{equation_g}, respectively. Then there holds
        \begin{equation}
          \label{posterior_estimate}
          \begin{split}
            h\norm{u-u_h}_{H^1(\Omega)}&\lesssim h\norm{u_h-q}_{L^2(\omega)}+J(u_h)^{1/2}+h^{3/2}\norm{Q\p_\nu u_h-\beta}_{L^2(\p\Omega)}\\
            &+[h^2\Delta u_h+h^2 f]_{H^0(\T_h)}+h\norm{z_h}_{H^1(\Omega)}\lesssim \trinorm{u-u_h,z_h}.
          \end{split}
        \end{equation}
      \end{theorem}
      \begin{proof}
        Applying Theorem \ref{lipschitz_stability}, we have
        \begin{equation}
          h\norm{u-u_h}_{H^1(\Omega)}\lesssim h\norm{u_h-q}_{L^2(\omega)}+h\norm{L^q_{u-u_h}}_{(H^1(\Omega))^*}.
        \end{equation}
        Here we recall $u|_\omega=q$. Using $-\Delta u=f$, $J(u)=0$, $Q\p_\nu u=\beta$ and Lemma \ref{lemma_laplacian}, we have
        \begin{equation}
          \begin{split}
          h\norm{L^q_{u-u_h}}_{(H^1(\Omega))^*}&\lesssim S(u-u_h)^{1/2}+B(u-u_h)^{1/2}+S^*(z_h)^{1/2}\\
          &\lesssim J(u_h)^{1/2}+h^{3/2}\norm{Q\p_\nu u_h-\beta}_{L^2(\p\Omega)}+[h^2\Delta u_h+h^2 f]_{H^0(\T_h)}+h\norm{z_h}_{H^1(\Omega)}.
        \end{split}
        \end{equation}
        Furthermore, notice that the right-hand side of the first inequality in \eqref{posterior_estimate} can be written as
      \begin{equation}
        h\norm{u-u_h}_{L^2(\omega)}+B(u-u_h)^{1/2}+J(u_h)^{1/2}+[h^2\Delta u_h+h^2 f]_{H^0(\T_h)}+h\norm{z_h}_{H^1(\Omega)}
      \end{equation}
      which is essentially identical to $\trinorm{u-u_h,z_h}$.
      \end{proof}
      Next we use the orthogonality and Lemma \ref{lemma_3} to derive the best approximation in the norm $\trinorm{\cdot}$. Before we present the lemma, we introduce the notation
      \begin{align}
        \norm{u}_{V_h^{-1/2}(\p\Omega)}=\sup_{w\in V_h}\frac{(u,w)_{L^2(\p\Omega)}}{\norm{w}_{H^1(\Omega)}}.
      \end{align}
      \begin{lemma}
        \label{lm:best_approximation}
        Suppose that $u\in H^{1}(\Omega)$ and $(u_h,z_h)\in V_h\times V_h$ solve \eqref{question} and \eqref{equation_g}, respectively. Then there holds
        \begin{equation}
          \label{best_approximation}
          \trinorm{u-u_h,z_h}\lesssim \inf_{\tilde{u}\in V_h}(\trinorm{u-\tilde{u},0}+\norm{h\nabla (u-\tilde{u})}_{L^2(\Omega)}+\norm{h\p_\nu(u-\tilde{u})}_{V_h^{-1/2}(\p\Omega)}).
        \end{equation}
      \end{lemma}
      \begin{proof}
        For any $\tilde{u}\in V_h$, we have
        \begin{equation}
          \trinorm{u_h-u,z_h}\leq \trinorm {u_h-\tilde{u},z_h}+\trinorm{\tilde{u}-u,0}.
        \end{equation}
        Applying Lemma \ref{lemma_3} and Galerkin orthogonality \eqref{Galerkin_orthogonality}, we get
        \begin{align}
          \trinorm{u_h-\tilde{u},z_h}&\lesssim \sup_{(v,w)\in V_h\times V_h}\frac{g(u_h-\tilde{u},z_h,v,w)}{\trinorm{v,w}}\\
          &=\sup_{(v,w)\in V_h\times V_h}\frac{g(u-\tilde{u},0,v,w)}{\trinorm{v,w}}.
        \end{align}
        Applying Cauchy-Schwarz inequality to each term in $g(u-\tilde{u},0,v,w)$, we have
        \begin{align}
            g(u-\tilde{u},0,v,w)&=h^2(u-\tilde{u},v)_{L^2(w)}+b(u-\tilde{u},v)+s(u-\tilde{u},v)+a(u-\tilde{u},w)\\
          &\leq \trinorm{u-\tilde{u},0}\cdot \trinorm{v,0}+a(u-\tilde{u},w).
        \end{align}
        Notice that for all $w\in V_h$, there holds
        \begin{align}
          &a(u-\tilde{u},w)=(h\nabla(u-\tilde{u}),h\nabla w)_{L^2(\Omega)}-h^2(P\p_\nu(u-\tilde{u}),w)_{L^2(\p\Omega)}\\
          &\lesssim \norm{h\nabla (u-\tu)}_{L^2(\Omega)}\cdot h\norm{w}_{H^1(\Omega)}+h\norm{P\p_\nu(u-\tilde{u})}_{V_h^{-1/2}(\p\Omega)}\cdot h\norm{w}_{H^1(\Omega)}\\
          &\lesssim \left(\norm{h\nabla (u-\tu)}_{L^2(\Omega)}+\norm{h P\p_\nu(u-\tilde{u})}_{V_h^{-1/2}(\p\Omega)}\right)\trinorm{0,w}.
        \end{align}
        Since $P$ is a projection operator, we have 
        \begin{align}
          \norm{hP \p_\nu(u-\tilde{u})}_{V_h^{-1/2}(\p\Omega)}\lesssim \norm{h \p_\nu(u-\tilde{u})}_{V_h^{-1/2}(\p\Omega)}.
        \end{align}
        Thus
        \begin{align}
        \label{eq:best_approximation_g}
          &g(u-\tilde{u},0,v,w)\lesssim \trinorm{u-\tilde{u},0}\cdot \trinorm{v,0}\\
          &+\left(\norm{h\nabla (u-\tilde{u})}_{L^2(\Omega)}+\norm{h \p_\nu(u-\tilde{u})}_{V_h^{-1/2}(\p\Omega)}\right)\trinorm{0,w}\\
          &\lesssim\left(\trinorm{u-\tilde{u},0}+\norm{h\nabla (u-\tilde{u})}_{L^2(\Omega)}+\norm{h \p_\nu(u-\tilde{u})}_{V_h^{-1/2}(\p\Omega)}\right)\trinorm{v,w}.
        \end{align}
        Hence for any $\tilde{u}\in V_h$,
        \begin{equation}
          \trinorm{u_h-u,z_h}\lesssim \trinorm{u-\tilde{u},0}+\norm{h\nabla (u-\tilde{u})}_{L^2(\Omega)}+\norm{h \p_\nu(u-\tilde{u})}_{V_h^{-1/2}(\p\Omega)},
        \end{equation}
        which implies \eqref{best_approximation}.
      \end{proof}
      Combining Theorem \ref{thm:aposterori} and Lemma \ref{lm:best_approximation}, we can obtain the a prior estimate
      \begin{theorem}
        \label{priori_estimate}
        Suppose that $u\in H^{1}(\Omega)$ and $(u_h,z_h)\in V_h \times V_h$ solves \eqref{question} and \eqref{equation_g}, respectively. Then there holds
        \begin{align}
         h\norm{u-u_h}_{H^1(\Omega)}\lesssim\inf_{\tilde{u}\in V_h}\left(\trinorm{u-\tilde{u},0}+\norm{h\nabla (u-\tilde{u})}_{L^2(\Omega)}+\norm{h \p_\nu(u-\tilde{u})}_{V_h^{-1/2}(\p\Omega)}\right).
        \end{align}
      \end{theorem}
      Next, we will show that Theorem \ref{priori_estimate} gives the optimal convergence rate.\par
      \begin{figure}[h!tbp]
    \centering
    \begin{tikzpicture}[scale=1]
    \node[fill=black,circle,scale=0.3] at (2,0) {};
    \node[fill=black,circle,scale=0.3] at (3,0) {};
    \node[fill=black,circle,scale=0.3] at (4,0) {};
    \node[fill=black,circle,scale=0.3] at (5,0) {};
    \node[fill=black,circle,scale=0.3] at (2.5,0.87) {};
    \node[fill=black,circle,scale=0.5] at (3.5,0.87) {};
    \node[fill=black,circle,scale=0.3] at (4.5,0.87) {};
    \node[fill=black,circle,scale=0.3] at (3,1.73) {};
    \node[fill=black,circle,scale=0.3] at (4,1.73) {};
    \draw[thick] (0,0) -- (7,0);
    \draw[thick] (2,0) -- (3,1.73);
    \draw[thick] (3,0) -- (4,1.73);
    \draw[thick] (4,1.73) -- (5,0);
    \draw[thick] (3,1.73) -- (4,0);
    \draw[thick] (3,1.73) -- (4,1.73);
    \draw[thick] (2.5,0.87) -- (4.5,0.87);
    \draw[thick] (2.5,0.87) -- (3,0);
    \draw[thick] (4.5,0.87) -- (4,0);
    \node[below left] at (3.5,0.87) {$x_T$};
    \node[above] at (3.5,1.73) {$T$};
    \node[below] at (6.5,0) {$\partial \Omega$};
    \draw[decoration={brace,mirror,raise=5pt},decorate] (2,0) -- node[below=6pt] {$F$} (5,0);
\end{tikzpicture}
  \caption{Patch $F$ on $\p\Omega$ together with associated bulk patch $T$. $\varphi_F|_K\in \mathbb{P}_1$ for all $K\in \T_h$, and $\varphi_F(x_T)=\Theta (h)$ and vanishes on other nodes.}  
  \label{fig:patch}
  \end{figure}
      As in \cite[Proposition 3.3]{BLO}, we decompose $\p\Omega$ in disjoint and shape regular patches $\{ F \}$ with diameter $\O(h)$ thus $|F| =\O(h^{d-1})$, and to each of them we associate a bulk patch $T$ that extends $\O(h)$ into $\Omega$ such that $T\cap \p \Omega=F$, and $T\cap T^\prime=\emptyset$ if $T\neq T^\prime$. On each patch $T$ we can construct a function $\varphi_F\in H^1_0(T)$ with $\supp \varphi_F\subset T$ satisfying
      \begin{equation}
        \norm{\p_\nu \varphi_F}_{L^\infty(F)}\leq C_\Omega,
      \end{equation}
      where $C_\Omega$ is a constant that only depends on $\Omega$, and
      \begin{equation}
         \int_F \p_\nu \varphi_F\ dx=|F|.
      \end{equation}
      When $h$ is small enough, we can take $T\cap \omega=\emptyset$ for all $T$ and $\varphi_F\in V_h$ such that $\varphi_F|_K\in \mathbb{P}_1$ for all $K\in \T_h$. A typical patch $F$ and $\varphi_F$ constructed on it is illustrated in Figure \ref{fig:patch}. We refer to \cite[Appendix]{BLO} for the construction in more general case. Then we define an interpolation $\pi_h:\ H^m(\Omega)\to V_h$:
      \begin{equation}
        \pi_h u=i_h u+\sum_{F}\alpha_F(u) \varphi_F,
      \end{equation}
      where $\alpha_F(u)$ is a functional defined as
      \begin{equation}
        \alpha_F(u)=\frac{1}{|F|}\int_F \p_\nu (u-i_h u)\ dx.
      \end{equation}
      Then we have the following lemma
\begin{lemma}
  \label{pi_estimate}
  Suppose $u\in H^{m}(\Omega)$, $m\geq 2$ and the polynomial order $k$ of the finite element space $V_h$ satisfies $k\geq m-1$, then there holds
  \begin{equation}
    h^{1/2}\norm{h\p_\nu(u-\pi_h u)}_{L^2(\p\Omega)}\lesssim [u]_{H^{m}(\Omega)}.
  \end{equation}
\end{lemma}
\begin{proof}
    By the definition of $\pi_h$, there holds
    \begin{equation}
      h^{3/2}\norm{\p_\nu(u-\pi_h u)}_{L^2(\p\Omega)}\leq h^{1/2}\norm{h\p_\nu (u-i_h u)}_{L^2(\p\Omega)}+h^{3/2}\norm{\sum_F \alpha_F(u)\p_\nu \varphi_F}_{L^2(\p\Omega)}.
    \end{equation}
    Notice that
    \begin{align}
      &\norm{\sum_F \alpha_F(u)\p_\nu \varphi_F}_{L^2(\p\Omega)}^2=\left(\sum_F \alpha_F(u)\p_\nu \varphi_F,\sum_{F^\prime}\alpha_{F^\prime}(u)\p_\nu\varphi_{F^\prime}\right)_{L^2(\p\Omega)}\\
      &=\sum_F\sum_{F^\prime}\alpha_F(u)\alpha_{F^\prime}(u)\left(\p_\nu\varphi_F,\p_\nu\varphi_{F^\prime}\right)_{L^2(\p\Omega)}=\sum_F \alpha_F^2(u)\norm{\p_\nu\varphi_F}_{L^2(F)}^2\\
      &\leq\sum_F \frac{C_\Omega}{|F|^2}\left(\int_F \p_\nu(u-i_h u)\ dx\right)^2\cdot |F|\leq \sum_F \frac{C_\Omega}{|F|}\int_F \mathbf{1}_Fdx\int_F|\p_\nu(u-i_h u)|^2dx\\
      &\lesssim \norm{\p_\nu (u-i_h u)}_{L^2(\p\Omega)}^2.
    \end{align}
    Thus
    \begin{align}
      &h^{3/2}\norm{\p_\nu(u-\pi_h u)}_{L^2(\p\Omega)}\lesssim h^{1/2}\norm{h\p_\nu (u-i_h u)}_{L^2(\p\Omega)}\\
      &\lesssim [u-i_h u]_{H^1(\T_h)}+[u-i_h u]_{H^2(\T_h)}\lesssim \norm{u-i_h u}_{H^m(\T_h)}\lesssim [u]_{H^m(\T_h)}.
    \end{align}
    Here the last inequality follows from the trace inequality and \eqref{interpolation}.
\end{proof}
\begin{proposition}
  \label{h_partial}
  Suppose $u\in H^{m}(\Omega)$, $m\geq 2$ and the polynomial order $k$ of the finite element space $V_h$ satisfies $k\geq m-1$, then there holds
  \begin{equation}
    \label{eq_h_partial}
    \norm{h \p_\nu(u-\pi_h u)}_{V_h^{-1/2}(\p\Omega)}\lesssim [u]_{H^m(\T_h)}.
  \end{equation}
\end{proposition}
\begin{proof}
  By definition, 
  \begin{align}
    \norm{h \p_\nu(u-\tilde{u})}_{V_h^{-1/2}(\p\Omega)}= \sup_{z\in V_h}\frac{(h\p_\nu (u-u_h),z)_{L^2(\p\Omega)}}{\norm{z}_{H^1(\Omega)}}.
  \end{align}
  By trace inequality, for all $z\in V_h$, there holds $z|_{\p\Omega}\in L^2(\p\Omega)$.
  Then we define an operator $\pi^0_h: L^2(\p\Omega)\to L^2(\p\Omega)$:
        \begin{equation}
          \pi^0_h(z)=\sum_{F}\left(\frac{1}{|F|}\int_F z\ dx\right)\mathbf{1}_F.
        \end{equation}
        Notice 
        \begin{equation}
          \int_{\p\Omega}z-\pi_h^0(z)=0,
        \end{equation}
        Applying the Poincar\'e's inequality on each $F$, we have
        \begin{align}
          &\norm{z-\pi^0_h(z)}_{L^2(\p\Omega)}^2\lesssim  h^2\norm{\nabla_\partial z}^2_{L^2(\p\Omega)}\lesssim h^2\norm{\nabla z}_{L^2(\p\Omega)}^2\lesssim \sum_{K\in \T_h}h^2\norm{\nabla z}^2_{L^2(\p K)}\\
          &\lesssim h\sum_{K\in \T_h}\left(\norm{\nabla z}^2_{L^2(K)}+\norm{h D^2 z}^2_{L^2(K)}\right)\lesssim h\norm{\nabla z}^2_{L^2(\Omega)}.
        \end{align}
        The last inequality follows from the discrete inverse inequality \eqref{eq:inverse}.
        Since 
        \begin{equation}
          \int_F \p_\nu(u-\pi_h u)dx=0
        \end{equation}
        for all $F$, there holds
        \begin{align}
          &(h\p_\nu(u-\pi_h u),z)_{L^2(\p\Omega)}=(h\p_\nu (u-\pi_h u),z-\pi_h^0 (z))_{L^2(\p\Omega)}\\
          &\lesssim h\norm{\p_\nu (u-\pi_h u)}_{L^2(\p\Omega)}\cdot\norm{z-\pi^0_h (z)}_{L^2(\p\Omega)}\lesssim h^{3/2} \norm{\p_\nu (u-\pi_h u)}_{L^2(\p\Omega)}\cdot\norm{z}_{H^1(\Omega)}
        \end{align}
        Then \eqref{eq_h_partial} follows immediately from Lemma \ref{pi_estimate}.
\end{proof}
\begin{proposition}
  \label{nabla_pi}
  Suppose $u\in H^{m}(\Omega)$, $m\geq 2$ and the polynomial order $k$ of the finite element space $V_h$ satisfies $k\geq m-1$, then there holds
  \begin{equation}
    \norm{h\nabla (u-\pi_h u)}_{L^2(\Omega)}\lesssim [u]_{H^m(\T_h)}.
  \end{equation}
\end{proposition}
\begin{proof}
  By definition
  \begin{equation}
    \norm{h\nabla(u-\pi_h u)}_{L^2(\Omega)}\leq \norm{h\nabla(u-i_h u)}_{L^2(\Omega)}+\norm{h \sum_{F}\alpha_F(u)\nabla\varphi_F}_{L^2(\Omega)}.
  \end{equation}
  Notice
  \begin{equation}
    \norm{h\nabla(u-i_h u)}_{L^2(\Omega)}\leq \norm{u-i_h u}_{H^2(\T_h)}\lesssim [u]_{H^2(\T_h)}.
  \end{equation}
  According to \cite[(3.2)]{BLO}, there holds
  \begin{equation}
    \norm{\nabla \varphi_F}_{L^2(\Omega)}\lesssim h^{\frac{d}{2}}.
  \end{equation}
  Thus
  \begin{align}
    &\norm{\sum_F \alpha_F(u) \nabla \varphi_F}_{L^2(\Omega)}^2\lesssim \sum_{F} \alpha_F^2(u)\norm{\nabla \varphi_F}_{L^2(\Omega)}^2\lesssim \sum_F h^d \alpha_F^2(u)\\
    &\lesssim h^d\sum_F \frac{1}{|F|^2}\left(\int_F \p_\nu(u-i_h u)\ dx\right)^2\lesssim h^d \sum_F \frac{1}{|F|} \int_F |\p_\nu(u-i_h u)|^2\ dx\\
    &\lesssim h^d\cdot h^{1-d} \sum_F \int_F|\p_\nu(u-i_h u)|^2\ dx\lesssim h\norm{\p_\nu (u-i_h u)}_{L^2(\p\Omega)}^2.
  \end{align}
  Applying Lemma \ref{pi_estimate}, we have
  \begin{align}
    &h\norm{\sum_F \alpha_F(u)\nabla\varphi_F}_{L^2(\Omega)}\lesssim h^{\frac{1}{2}}\norm{h\p_\nu(u-i_h u)}_{L^2(\p\Omega)}\lesssim [u]_{H^m(\T_h)}.
  \end{align}
\end{proof}
\begin{proposition}
  \label{trinorm}
  Suppose $u\in H^{m}(\Omega)$, $m\geq 2$ and the polynomial order $k$ of the finite element space $V_h$ satisfies $k\geq m-1$, then there holds
  \begin{equation}
    \trinorm{u-\pi_h u,0}\lesssim [u]_{H^m(\T_h)}.
  \end{equation}
\end{proposition}
\begin{proof}
  By definition,
  \begin{equation}
    \trinorm{u-\pi_h u,0}\lesssim B(u-\pi_h u)^{1/2}+h\norm{u-\pi_h u}_{L^2(\omega)}+S(u-\pi_h u)^{1/2}.
  \end{equation}
  Since for the patch $T$ introduced before Lemma \ref{pi_estimate} there holds $T\cap \omega=\emptyset$,
  \begin{equation}
    \norm{u-\pi_h u}_{L^2(\omega)}=\norm{u-i_h u}_{L^2(\omega)}\lesssim \norm{u-i_h u}_{H^m(\T_h)}\lesssim [u]_{H^m(\T_h)}.
  \end{equation}
  And there holds
  \begin{align}
    B(u-\pi_h u)^{1/2}&=h^{\frac{1}{2}}\norm{hQ\p_\nu(u-\pi_h u)}_{L^2(\p\Omega)}\lesssim h^{\frac{1}{2}}\norm{h\p_\nu (u-\pi_h u)}_{L^2(\p\Omega)}\\
    &\leq[u-\pi_h u]_{H^1(\T_h)}+[u-\pi_h u]_{H^2(\T_h)}.
  \end{align}
  Notice $D^2 \varphi_F=0$, then there holds
  \begin{equation}
    \label{eq:H2}
    [u-\pi_h u]_{H^2(\T_h)}=[u-i_h u]_{H^2(\T_h)}.
  \end{equation}
  Applying Proposition \ref{nabla_pi}, we have
  \begin{equation}
    B(u-\pi_h u)^{1/2}\lesssim [u]_{H^m(\T_h)}+\norm{u-i_h u}_{H^m(\T_h)}\lesssim [u]_{H^m(\T_h)}.
  \end{equation}
  Finally, according to \eqref{eq:trace_inequality}, \eqref{eq:H2} and Proposition \ref{nabla_pi}, 
  \begin{align}
    S(u-\pi_h u)^{1/2}&\lesssim [u-\pi_h u]_{H^1(\T_h)}+[u-\pi_h u]_{H^2(\T_h)}\\
    &\lesssim \norm{u-i_h u}_{H^m(\T_h)}\lesssim [u]_{H^m(\T_h)}.
  \end{align}
\end{proof}
Combining Theorem \ref{priori_estimate} and the above three propositions, we can conclude the following corollary
\begin{corollary}
  \label{optimal_rate}
  Suppose $u\in H^{m}(\Omega)$ and $(u_h,z_h)\in V_h\times V_h$ solves \eqref{question} and \eqref{equation_g}, respectively. Here $m\geq 2$ and the polynomial order $k$ of the finite element space $V_h$ satisfies $k\geq m-1$. Then there holds
  \begin{equation}
    \norm{u-u_h}_{H^1(\Omega)}\lesssim h^{m-1}\norm{D^{m}u}_{L^2(\Omega)}.
  \end{equation}  
\end{corollary}
\section{Perturbation analysis}
Consider the finite element method with perturbation $q_\delta\in L^2(\omega)$,
\begin{equation}
  \label{eq:perturbed_g}
  \begin{split}
    g(u,z,v,w)&=h^2(f,w)_{L^2(\Omega)}+h^2(\beta,w)_{L^2(\p\Omega)}+h^2(q+q_\delta,v)_{L^2(\omega)}\\
    &+b(\beta,v)+(h^2f,h^2\Delta v)_{H^0(\T_h)}.
  \end{split}
\end{equation}
where $\norm{q_\delta}_{L^2(\omega)}\leq \delta$. We have the following theorem.
\begin{theorem}
  Suppose $\Omega$ is convex, $u\in H^{2}(\Omega)$ satisfies
  \begin{equation}
    \begin{cases}
      -\Delta u=f\text{ in }\Omega,\\
      u=q\text{ in }\omega,
    \end{cases}
  \end{equation}
  and there exists $p\in \mathcal{V}_N+\beta$ such that $\norm{\p_\nu u-p}_{H^{1/2}(\p\Omega)}\leq \delta$. Let $(u_h,z_h)\in V_h\times V_h$ be the solution of \eqref{eq:perturbed_g} with $k=1$ in $V_h$. Then there holds
  \begin{equation}
    \label{eq:estimate_perturbation}
    \norm{u-u_h}_{H^1(\Omega)}\lesssim h\norm{D^2 u}_{L^2(\Omega)}+\delta.
  \end{equation}
\end{theorem}

\begin{proof}
  Let $u^*\in H^1(\Omega)$ solve 
  \begin{equation}
    \begin{cases}
      -\Delta u^*=f\text{ in }\Omega,\\
      \p_\nu u^* =p\text{ on }\p\Omega,\\
      \int_{\omega}u^*\ dx=\int_{\omega}u\ dx,
    \end{cases}
  \end{equation}
  and $u^\delta=u-u^*$. Notice that $u^\delta\in H_\omega^1(\p\Omega)$ and satisfies
  \begin{equation}
    \begin{cases}
      -\Delta u^\delta=0\text{ in }\Omega,\\
      \p_\nu u^\delta=\p_\nu u-p\text{ on }\p\Omega.
    \end{cases}
  \end{equation}
  Denote $\p_\nu u-p$ by $p_\delta$. According to the elliptic regularity \cite[Corollary 2.2.2.6]{Grisvard_Ell} and Lemma \ref{lemma_1}, there holds
  \begin{align}
    \norm{u^\delta}_{H^2(\Omega)}\lesssim \norm{u^\delta}_{L^2(\Omega)}+\norm{p_\delta}_{H^{1/2}(\p\Omega)}\lesssim \norm{q_\delta}_{L^2(\omega)}+\norm{p_\delta}_{H^{1/2}(\p\Omega)}\lesssim \delta.
  \end{align}
  Since $\p_\nu u^*\in\mathcal{V}_N+\beta$, we can obtain the finite element approximation $(u_h^*,z_h^*)$ of $u^*$ by the method \eqref{equation_g}. Then Corollary \ref{optimal_rate} says that
  \begin{align}
    \norm{u^*-u^*_h}_{H^1(\Omega)}&\lesssim h\norm{D^2 u^*}_{L^2(\Omega)}\lesssim h\norm{D^2 u}_{L^2(\Omega)}+h\norm{D^2 u^\delta}_{L^2(\Omega)}\\
    &\lesssim h\norm{D^2 u}_{L^2(\Omega)}+h\delta.
  \end{align}
  and 
  \begin{align}
    \norm{u-u_h}_{H^1(\Omega)}&\leq \norm{u^\delta}_{H^1(\Omega)}+\norm{u^*-u^*_h}_{H^1(\Omega)}+\norm{u^*_h-u_h}_{H^1(\Omega)}\\
    &\lesssim \delta + h\norm{D^2 u}_{L^2(\Omega)}+\norm{u^*_h-u_h}_{H^1(\Omega)}.
  \end{align}
  Notice that
  \begin{equation}
    g(u_h-u_h^*,z_h-z_h^*,v,w)=h^2(q_\delta+u^\delta, v)_{L^2(\omega)}.
  \end{equation}
  By Lemma \ref{lemma_3}, we have
  \begin{align}
    &\trinorm{u_h-u_h^*,z_h-z_h^*}\lesssim \sup_{(v,w)\in V_h}\frac{g(u_h-u_h^*,z_h-z_h^*,v,w)}{\trinorm{v,w}}\\
    &\lesssim h\sup_{(v,w)\in V_h}\frac{(q_\delta+u^\delta, v)_{L^2(\omega)}}{\norm{v}_{L^2(\omega)}}\lesssim h(\norm{q_\delta}_{L^2(\omega)}+\norm{u^\delta}_{L^2(\omega)})\lesssim h\delta.
  \end{align}
  Combining Theorem \ref{lipschitz_stability} and Lemma \ref{a_c_estimate}, there holds
  \begin{align}
    \norm{u_h-u_h^*}_{H^1(\Omega)}\lesssim h^{-1}\trinorm{u_h-u_h^*,0}\lesssim \delta,
  \end{align}
  which concludes \eqref{eq:estimate_perturbation}
\end{proof}

\section{The approximation of flux}
In this section, we will propose a modified method which can be used to obtain a finite element approximation that possesses a convergence rate of the flux in natural norm on the boundary.\par
Firstly we introduce the Lagrangian $\widetilde {\mathcal{L}}$ on $V_h\times V_h \times V_h$
\begin{align}
  \widetilde {\mathcal{L}}(u,z,r)&:=\mathcal{L}(u,z)+\ta(u,r)-h^2(f,r)_{L^2(\Omega)}-\frac{1}{2}S^*(r),
\end{align}
where
\begin{equation}
  \ta(u,r)=(h\nabla u,h\nabla r)_{L^2(\Omega)}-h\int_{\p\Omega}h (\p_\nu u)r\ dx.
\end{equation}
Analogous to Lemma \ref{a_c_estimate}, there holds
\begin{equation}
  \label{eq:tri_ta}
  \ta(u,r)\lesssim S(u)^{1/2}\norm{r}_{H^1(\T_h)}.
\end{equation}
Finding saddle points of the Lagrangian $\widetilde{\mathcal{L}}$ amounts to finding $(u,z,r)\in V_h\times V_h \times V_h$ such that for all $(v,w,t)\in V_h\times V_h \times V_h$, there holds
      \begin{align}
        \label{equation_tg}
        \tg(u,z,r,v,w,t)&=h^2(f,w)_{L^2(\Omega)}+h^2(\beta,w)_{L^2(\p\Omega)}+h^2(f,t)_{L^2(\Omega)}\\
        &+h^2(q,v)_{L^2(\omega)}+h^3(Q \p_\nu v, \beta)_{L^2(\p\Omega)}+(h^2f,h^2\Delta v)_{H^0(\T_h)},
      \end{align}
      where the bilinear form $\tg$ is defined by
      \begin{align}
        \label{eq:def_tg}
        \tg(u,z,r,v,w,t)&:=g(u,z,v,w)+\ta(v,r)+\ta(u,t)-s^*(r,t).
      \end{align}
Similarly, if $u\in H^{1}(\Omega)$ solves \eqref{question}, then $(u,0,0)$ solves \eqref{equation_tg} for all $(v,w,t)\in V_h\times V_h\times V_h$. Thus the system \eqref{equation_tg} is consistent. Furthermore, since 
\begin{equation}
  \trinorm{u,z}+S^*(r)=\tg(u,z,r,u,-z,-r),
\end{equation}
there holds
\begin{equation}
  \label{eq:tri_uniqueness}
  \trinorm{u,z}+S^*(r)^{1/2}\lesssim \sup_{(v,w,t)\in V_h\times V_h\times V_h}\frac{\tg(u,z,r,v,w,t)}{\trinorm{v,w}+S^*(t)^{1/2}}.
\end{equation}
Then it is inferable that \eqref{equation_tg} admits a unique solution. Denoting the unique solution by $(\tu_h,\tz_h,r_h)$ and letting $u$ solve \eqref{question}, we have the Galerkin orthogonality
\begin{equation}
  \label{tri_Galerkin_orthogonality}
  \tg(\tu_h-u,\tz_h,r_h,v,w,t)=0,\ (v,w,t)\in V_h\times V_h\times V_h.
\end{equation}
Replacing $u_h$ by $\tu_h$ and $(0,i_h w)$ by $(0,i_h w,0)$ in the proof for Lemma \ref{lemma_laplacian}, we can prove that
\begin{equation}
  \label{eq:tri_laplacian}
  h^{-1}\norm{h^2 L^q_{u-\tu_h}}_{(H^1(\Omega))^*}\lesssim S(u-\tu_h)^{1/2}+B(u-\tu_h)^{1/2}+S^*(\tz_h)^{1/2}.
\end{equation}
Then we can obtain a posteriori error estimate of the flux in natural norm on the boundary.
  \begin{theorem}
    \label{flux_posteriori}
      Suppose that $u\in H^{1}(\Omega)$ and $(\tu_h,\tz_h,r_h)\in V_h$ solves \eqref{question} and \eqref{equation_tg}, respectively. Then there holds
    \begin{equation}
      \label{eq:posteriori_bd}
          \begin{split}
            h\norm{u-\tu_h}_{H^1(\Omega)}&+h\norm{\p_\nu (u-\tu_h)}_{H^{-1/2}(\p\Omega)}\lesssim h\norm{\tu_h-q}_{L^2(\omega)}+h^{3/2}\norm{Q\p_\nu \tu_h-\beta}_{L^2(\p\Omega)}\\
          &+J(\tu_h)^{1/2}+[h^2\Delta \tu_h+h^2 f]_{H^0(\T_h)}+h\norm{\tz_h}_{H^1(\Omega)}+h\norm{r_h}_{H^1(\Omega)}.
          \end{split}
    \end{equation}
  \end{theorem}
  \begin{proof}
    By definition
    \begin{align}
      &h^2\norm{\p_\nu (u-\tu_h)}_{H^{-1/2}(\p\Omega)}=\sup_{t\in H^{1/2}(\p\Omega)}\frac{h^2(\p_\nu (u-\tu_h),t)_{L^2(\p\Omega)}}{\norm{t}_{H^{1/2}(\p\Omega)}}\\
      &\lesssim\sup_{t\in H^{1}(\Omega)}\frac{h^2(\p_\nu (u-\tu_h),t)_{L^2(\p\Omega)}}{\norm{t}_{H^{t}(\p\Omega)}}=\sup_{t\in H^{1}(\Omega)}\frac{\ta(u-\tu_h,t)-(h\nabla(u-\tu_h),h\nabla t)_{L^2(\Omega)}}{\norm{t}_{H^{1}(\p\Omega)}}\\
      &\lesssim \sup_{t\in H^{1}(\Omega)}\frac{\ta(u-\tu_h,t)}{\norm{t}_{H^{1}(\p\Omega)}}+h^2\norm{u-\tu_h}_{H^1(\Omega)}.
    \end{align}
    Taking $(v,w,t)=(0,0,i_h t)$ in \eqref{tri_Galerkin_orthogonality}, we have
    \begin{equation}
      \tg(\tu_h-u,\tz_h,r_h,0,0,i_h t)=\ta(u_h-u,i_h t)-s^*(r_h,i_h t)=0.
    \end{equation}
    By \eqref{eq:tri_ta}, there holds
    \begin{align}
      \ta(u-\tu_h,t)&=\ta(u-\tu_h,t-i_h t)-s^*(r_h,i_h t)\\
      &\lesssim S(u-\tu_h)^{1/2}\norm{t-i_h t}_{H^1(\T_h)}+S^*(r_h)^{1/2}S^*(i_h t)^{1/2}\\
      &\lesssim S(u-\tu_h)^{1/2}[t]_{H^1(\T_h)}+S^*(r_h)^{1/2}h\left(\norm{i_h t-t}_{H^1(\Omega)}+\norm{t}_{H^1(\Omega)}\right)\\
      &\lesssim h \left(J^{1/2}(\tu_h)+[h^2\Delta \tu_h+h^2 f]_{H^0(\T_h)}+S^*(r_h)^{1/2}\right)\norm{t}_{H^1(\Omega)}.
    \end{align}
    Therefore, we have
    \begin{equation}
      \begin{split}
        h^2\norm{\p_\nu (u-\tu_h)}_{H^{-1/2}(\p\Omega)}&\lesssim h J(\tu_h)^{1/2}+h[h^2\Delta \tu_h+h^2 f]_{H^0(\T_h)}\\
        &+h^2\norm{r_h}_{H^1(\Omega)}+h^2\norm{u-\tu_h}_{H^1(\Omega)}
      \end{split}
    \end{equation}
    Applying Theorem \ref{lipschitz_stability} and \eqref{eq:tri_laplacian}, there holds
    \begin{align}
      &h\norm{u-\tu_h}_{H^1(\Omega)}\lesssim h\norm{\tu_h-q}_{L^2(\omega)}+h\norm{L^q_{u-\tu_h}}_{(H^1(\Omega))^*}\\
      & \lesssim h\norm{\tu_h-q}_{L^2(\omega)}+h^{3/2}\norm{Q\p_\nu \tu_h-\beta}_{L^2(\p\Omega)}+J(\tu_h)^{1/2}+[h^2\Delta \tu_h+h^2 f]_{H^0(\T_h)}+h\norm{\tz_h}_{H^1(\Omega)},
    \end{align}
    which leads to \eqref{eq:posteriori_bd}.
  \end{proof}\par
  Observe that $\trinorm{u-\tu_h,\tz_h}+S^*(r_h)^{1/2}$ is essentially the same as the right-hand side of \eqref{eq:posteriori_bd}. Similarly to Section \ref{sec:FEM_error}, we derive the following best approximation for it.
      \begin{lemma}
        \label{lm:flux_best_approximation}
        Suppose that $u\in H^{1}(\Omega)$ and $(u_h,z_h,r_h)\in V_h\times V_h\times V_h$ solve \eqref{question} and \eqref{equation_g}, respectively. Then there holds
        \begin{equation}
          \label{eq:tri_best_approximation}
          \trinorm{u-\tu_h,\tz_h}+S^*(r_h)^{1/2}\lesssim \inf_{\tilde{u}\in V_h}(\trinorm{u-\tilde{u},0}+\norm{h\nabla (u-\tilde{u})}_{L^2(\Omega)}+\norm{h\p_\nu(u-\tilde{u})}_{V_h^{-1/2}(\p\Omega)}).
        \end{equation}
      \end{lemma}
      \begin{proof}
        For any $\tu\in V_h$, applying \eqref{eq:def_tg}, \eqref{eq:tri_uniqueness} and \eqref{tri_Galerkin_orthogonality}, there holds
        \begin{align}
          \trinorm{u-\tu_h,\tz_h}&+S^*(r_h)^{1/2}\leq \trinorm{\tu_h-\tu,\tz_h}+S^*(r_h)^{1/2}+\trinorm{\tu-u,0}\\
          &\lesssim \sup_{(v,w,t)\in V_h\times V_h\times V_h}\frac{\tg(\tu_h-\tu,\tz_h,r_h,v,w,t)}{\trinorm{v,w}+S^*(t)^{1/2}}+\trinorm{\tu-u,0}\\
          &\lesssim \sup_{(v,w,t)\in V_h\times V_h\times V_h}\frac{\tg(u-\tu,0,0,v,w,t)}{\trinorm{v,w}+S^*(t)^{1/2}}+\trinorm{\tu-u,0}\\
          &\lesssim \sup_{(v,w)\in V_h\times V_h}\frac{g(u-\tu,0,v,w)}{\trinorm{v,w}}+\sup_{t\in V_h}\frac{\ta(u-\tu,t)}{S^*(t)^{1/2}}+\trinorm{\tu-u,0}.
        \end{align}
        We have already shown in \eqref{eq:best_approximation_g} that for all $\tu\in V_h$ and $(v,w)\in V_h\times V_h$, there holds  
        \begin{align}
          \frac{g(u-\tu,0,v,w)}{\trinorm{v,w}}\lesssim\trinorm{u-\tilde{u},0}+\norm{h\nabla (u-\tilde{u})}_{L^2(\Omega)}+\norm{h\p_\nu(u-\tilde{u})}_{V_h^{-1/2}(\p\Omega)}.
        \end{align}
        Finally, for $t\in V_h$, we have
        \begin{align}
          &\ta(u-\tilde{u},t)=(h\nabla(u-\tilde{u}),h\nabla t)_{L^2(\Omega)}-h^2(P\p_\nu(u-\tilde{u}),t)_{L^2(\p\Omega)}\\
          &\lesssim \norm{h\nabla (u-\tu)}_{L^2(\Omega)}\cdot h\norm{t}_{H^1(\Omega)}+h\norm{\p_\nu(u-\tilde{u})}_{V_h^{-1/2}(\p\Omega)}\cdot h\norm{t}_{H^1(\Omega)}\\
          &\lesssim \left(\norm{h\nabla (u-\tu)}_{L^2(\Omega)}+\norm{h\p_\nu(u-\tilde{u})}_{V_h^{-1/2}(\p\Omega)}\right)S^*(t)^{1/2}.
        \end{align}
      \end{proof}
      Observe that the right-hand side of \eqref{eq:tri_best_approximation} coincides with the right-hand side of \eqref{best_approximation}, then the corollary below follows from the same arguments in Section \ref{sec:FEM_error}.
      \begin{corollary}
        \label{flux_optimal_rate}
        Suppose $u\in H^{m}(\Omega)$ and $(\tu_h,\tz_h,r_h)\in V_h\times V_h\times V_h$ solves \eqref{question} and \eqref{equation_tg}, respectively. Here $m\geq 2$ and the polynomial order $k$ of the finite element space $V_h$ satisfies $k\geq m-1$. Then there holds
        \begin{equation}
          \norm{u-\tu_h}_{H^1(\Omega)}+\norm{\p_\nu (u-\tu_h)}_{H^{-1/2}(\p\Omega)}\lesssim h^{m-1}\norm{D^{m}u}_{L^2(\Omega)}.
        \end{equation}  
      \end{corollary}

\section{Numerical results}

In order to simplify the implementation of the finite element method, in the numerical examples below, we take $\Omega = [0,1]^2$ and $\omega = [0.1,0.9]\times[0.25, 0.75]$. We denote $\Gamma \in \p\Omega$ the top edge of the unit square, that is
  \begin{equation}
    \Gamma =\{(x,y)\mid 0<x<1\text{ and }y=1\},
  \end{equation}
  and let $\mathcal{V}_N$ be the subspace of $L^2(\p\Omega)$ which is spanned by the basis $\{\phi_n\},\ 1\leq n\leq N$, where $\phi_n$ is $\sqrt{2}\cos(n\pi x)$ on $\Gamma$ and vanishes on $\p\Omega\backslash\Gamma$. Furthermore, the polynomial order of the finite element space $V_h$ is one.\par
  Our first computational example studies the effect of the stabilizing term $s$ in the computations. Without making any theoretical difference, we rescale $s$ by a positive constant $\gamma>0$.
  \begin{figure}[h!tbp]
    \centering
    \begin{subfigure}{0.45\textwidth}
     \centering
    \includegraphics[width=\textwidth]{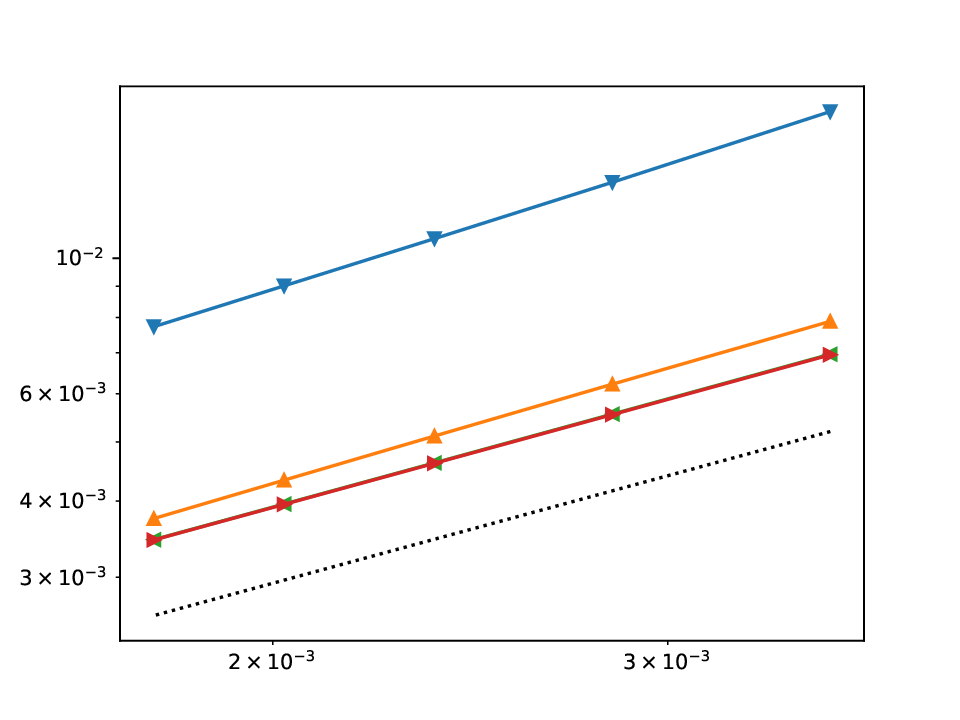}
    \subcaption{}
    \label{gamma}
    \end{subfigure}
    \begin{subfigure}{0.45\textwidth}
     \centering
    \includegraphics[width=\textwidth]{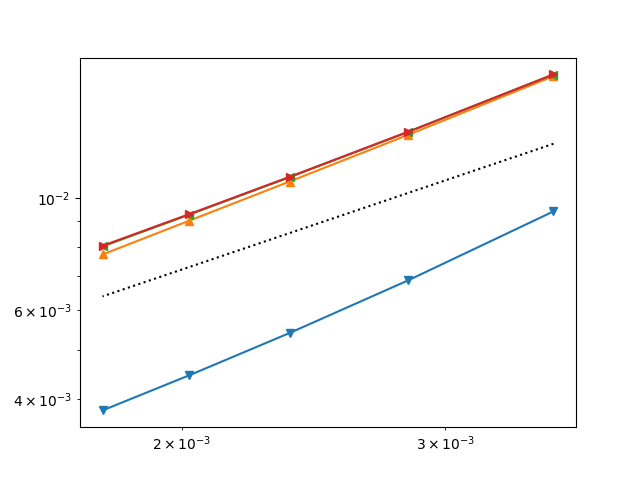}
    \subcaption{}
    \label{dimension}
    \end{subfigure}
    \caption{The error $\norm{u-u_h}_{H^1}$ as a function of mesh size $h$ with the reference rate $h$ presented by the dash line.  (a) $u_h$ is computed with parameter $\gamma = 1,\ 10^{-1},\ 10^{-2},\ 0$ with triangles pointing down, up, left, and right, respectively. (b) $u_h$ is obtained with different $\mathcal{V}_N$ where $N = 1,\ 8,\ 16,\ 64$ with triangles pointing down, up, left, and square, respectively.}
  \end{figure}
  Choosing $N = 8$ for space $\mathcal{V}_N$, and the $H^1(\Omega)$ error for the artificial solution 
  \begin{equation}
    \label{example_1}
    u(x,y) = (e^y-y)\cos(\pi x)
  \end{equation}
  as a function of the mesh parameter $h$ is illustrated for different values of $\gamma$ in Figure \ref{gamma}. It shows that the numerical results improve as $\gamma\to 0$, which implies that in this example the stabilizing term $s$ can be omitted as the method still converge with the optimal rate without it.\par
  Next, we study the effect of the dimension $N$ of $\mathcal{V}_N$ in the computations, the $H^1(\Omega)$ for $u(x,y)$ in \eqref{example_1} as a function of the mesh size $h$ is shown for different $N$ in Figure \ref{dimension}. It appears that when $N$ is large enough, merely increasing $N$ can not affect the results.\par
  \begin{figure}[h!tbp]
  \centering
    \includegraphics[width=0.5\linewidth]{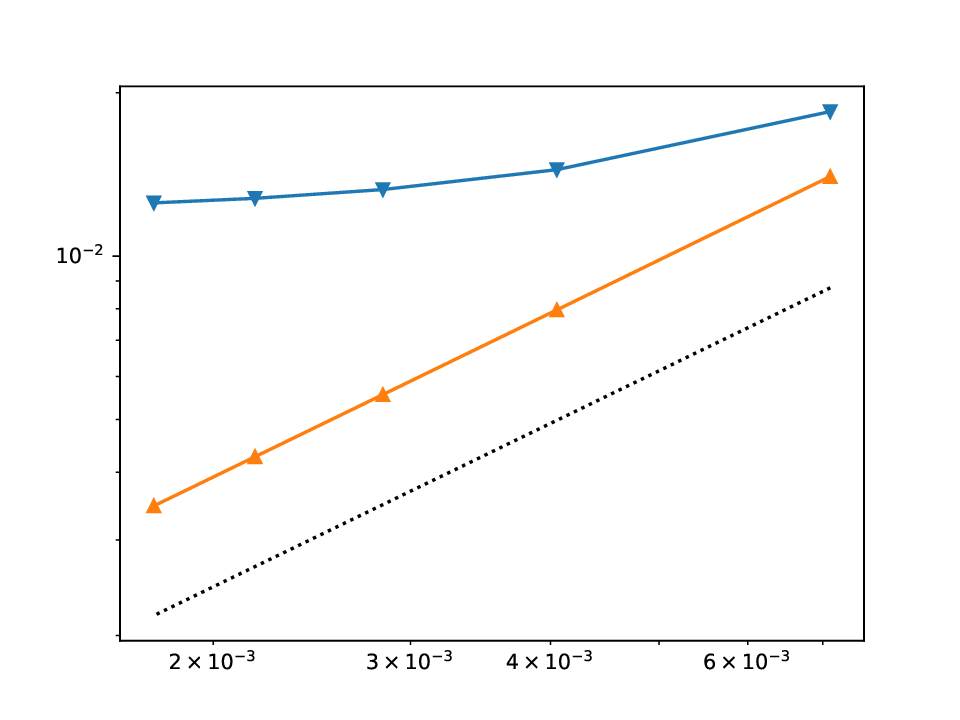}
    \caption{The error $\norm{u-u_h}_{H^1}$ as a function of mesh size $h$. Here $N=1,2$ with triangles pointing down and up. Reference rate $h$ is presented by the dash line.}
    \label{perturbation}
  \end{figure}
  Our second example illustrates the error estimate with perturbations. Taking 
  \begin{equation}
    \tilde{u}(x,y)=(e^y-y)\cos(\pi x)+0.025(e^y-y)\cos(2\pi x),
  \end{equation}
  Figure \ref{perturbation} shows the stagnation of the convergence for $N=1$ but not for $N=2$ because $\p_\nu \tilde{u}|_{\p\Omega}\in \mathcal{V}_2$, which is compatible with our theory.\par

  Then we study the effect of random perturbations of data. To this end, we denote the right-hand side of the equation \eqref{equation_g} by $F(v,w)$. Instead of solving \eqref{equation_g}, we numerically solve 
  \begin{equation}
      g(u,z,v,w)=F(v,w)+\varepsilon \delta \frac{\abs{F(v,w)}}{\abs{\delta}},
  \end{equation}
  where the random variable $\delta \sim \mathcal{N}(0,1)$ and $\mathcal{N}(0,1)$ represents the standard normal distribution. Figure \ref{noise} shows the stagnation of the convergence for $\varepsilon\neq 0$ compared with the consistent convergence without white noise.

\begin{figure}[h!tbp]
  \centering
    \includegraphics[width=0.5\linewidth]{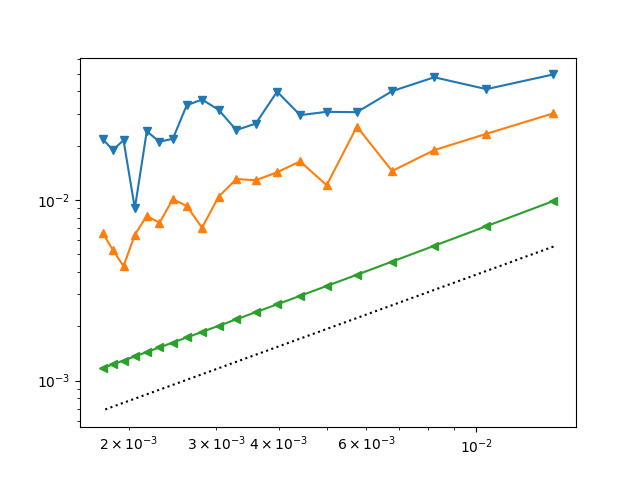}
    \caption{The error $\norm{u-u_h}_{H^1}$ as a function of mesh size $h$. Here $\varepsilon= 0.12, 0.06,0$ with triangles pointing down, up and left. Reference rate $h$ is presented by the dash line.}
    \label{noise}
  \end{figure}

  To close this section, we define the ratio
  \begin{equation}
    C(u) = \frac{\norm{u-u_h}_{H^1(\Omega)}}{h\norm{u}_{H^2(\Omega)}},
  \end{equation}
  and consider the functions
  \begin{equation}
    \label{example_2}
    u_N(x,y) = (e^y-y)\cos(N\pi x),\ N=1,\ 2,\ 3,\ 4.
  \end{equation}
  \begin{figure}[h!tbp]
    \centering
    \begin{subfigure}{0.45\textwidth}
     \centering
    \includegraphics[width=\textwidth]{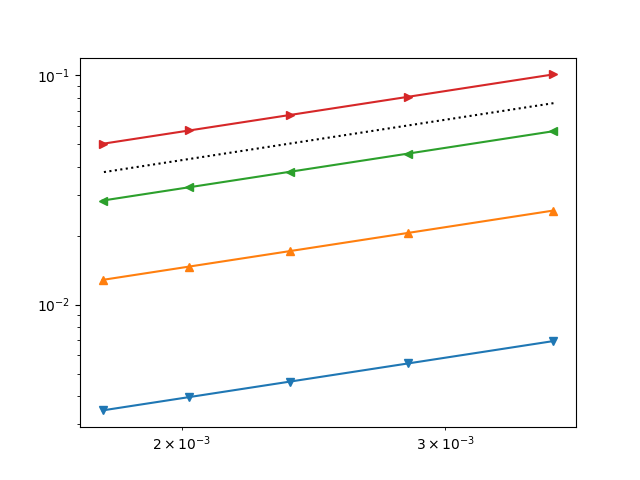}
    \subcaption{}
    \label{function_difference}
    \end{subfigure}
    \begin{subfigure}{0.45\textwidth}
     \centering
    \includegraphics[width=\textwidth]{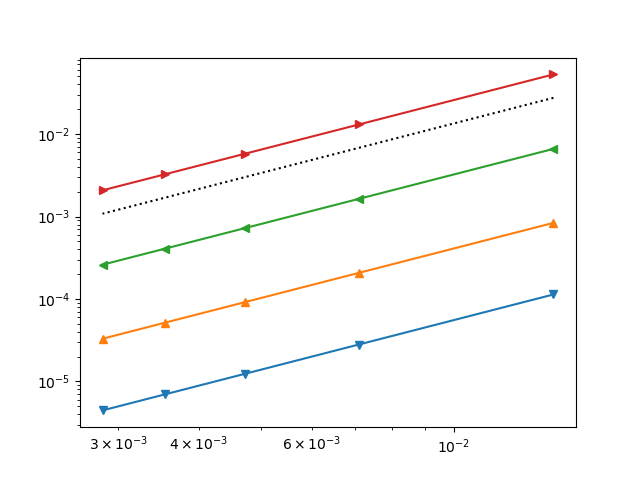}
    \subcaption{}
    \label{second_order}
    \end{subfigure}
    \caption{(a) The error $\norm{u_N-u_{Nh}}_{H^1(\Omega)}$ as a function of mesh size $h$ where $u_{Nh}$ is computed by first-order method. Here $N = 1,\ 2,\ 3,\ 4$ with the triangles pointing down, up, left and right, respectively. Reference rate $h$ in dashed. (b) The error $\norm{u_N-u_{Nh}}_{H^1(\Omega)}$ as a function of mesh size $h^2$ where $u_{Nh}$ is computed by second-order method. Here $N = 1,\ 2,\ 3,\ 4$ with the triangles pointing down, up, left and right, respectively. Reference rate $h^2$ in dashed.}
  \end{figure}
  We use both the first-order and second-order finite element methods to compute $u_{Nh}$.
  The $H^1(\Omega)$ errors between $u_N$ and $u_{Nh}$ with first-order and second-order methods are plotted in Figure \ref{function_difference} and \ref{second_order}, respectively. \par
  Besides, the ratio $C(u)$ for $u$ as in \eqref{example_1}, and for $u_N$ is shown in Figure \ref{ratio} as a function of $N$. For any $u$, the ratio $C(u)$ gives a lower bound for the constant $C>0$ in our priori estimate. It appears that the constant grows as a function of $N$, as expected.
\begin{figure}[h!tbp]
  \centering
    \includegraphics[width=0.5\linewidth]{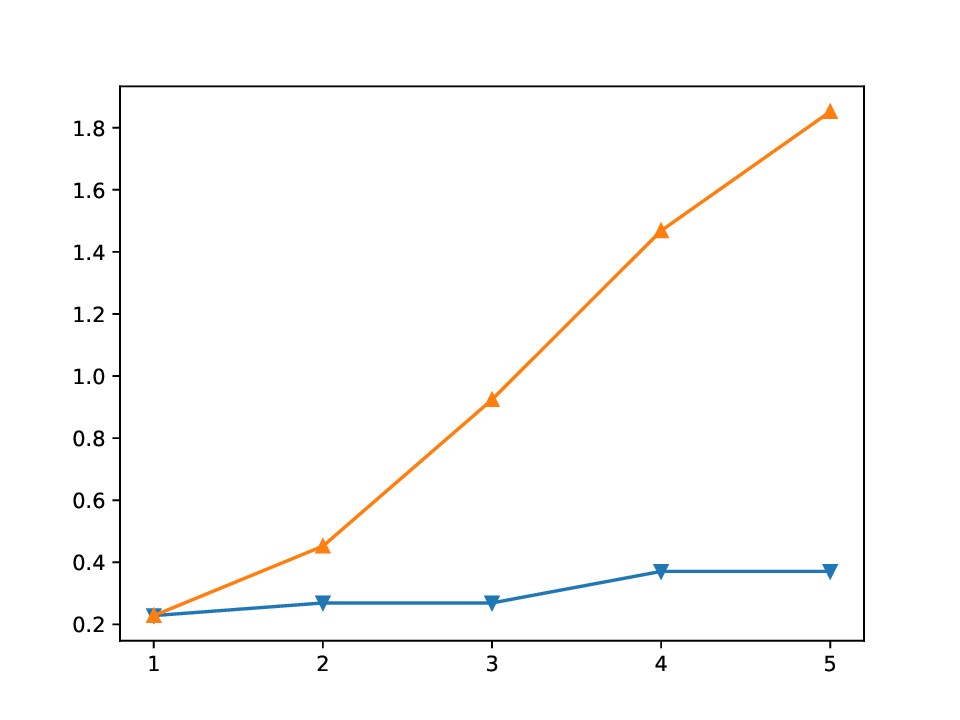}
    \caption{The ratio $C(u)$ as a function of $N$. Here $u$ is as in \eqref{example_1} and \eqref{example_2} with triangles pointing down and up, respectively.}
    \label{ratio}
  \end{figure}

\clearpage

\bibliographystyle{siamplain}
\bibliography{references}

\begin{thebibliography}{10}

\bibitem{Giovanni_2009}
{\sc G.~Alessandrini, L.~Rondi, E.~Rosset, and S.~Vessella}, {\em The stability for the {Cauchy} problem for elliptic equations}, Inverse Probl., 25 (2009), p.~47, \url{https://doi.org/10.1088/0266-5611/25/12/123004}.

\bibitem{AV05}
{\sc G.~Alessandrini and S.~Vessella}, {\em Lipschitz stability for the inverse conductivity problem}, Adv. in Appl. Math., 35 (2005), pp.~207--241, \url{https://doi.org/10.1016/j.aam.2004.12.002}, \url{https://doi.org/10.1016/j.aam.2004.12.002}.

\bibitem{belgacem2022}
{\sc F.~Ben~Belgacem, V.~Girault, and F.~Jelassi}, {\em Full discretization of {Cauchy}'s problem by {Lavrentiev}-finite element method}, SIAM J. Numer. Anal., 60 (2022), pp.~558--584, \url{https://doi.org/10.1137/21M1401310}.

\bibitem{Bou05}
{\sc L.~Bourgeois}, {\em A mixed formulation of quasi-reversibility to solve the {C}auchy problem for {L}aplace's equation}, Inverse Problems, 21 (2005), pp.~1087--1104, \url{https://doi.org/10.1088/0266-5611/21/3/018}, \url{https://doi.org/10.1088/0266-5611/21/3/018}.

\bibitem{Bour13}
{\sc L.~Bourgeois}, {\em A remark on {L}ipschitz stability for inverse problems}, C. R. Math. Acad. Sci. Paris, 351 (2013), pp.~187--190, \url{https://doi.org/10.1016/j.crma.2013.04.004}, \url{https://doi.org/10.1016/j.crma.2013.04.004}.

\bibitem{BC20}
{\sc L.~Bourgeois and L.~Chesnel}, {\em On quasi-reversibility solutions to the {C}auchy problem for the {L}aplace equation: regularity and error estimates}, ESAIM Math. Model. Numer. Anal., 54 (2020), pp.~493--529, \url{https://doi.org/10.1051/m2an/2019073}, \url{https://doi.org/10.1051/m2an/2019073}.

\bibitem{laurent2022}
{\sc L.~Bourgeois and J.~Dard{\'e}}, {\em The {Morozov}'s principle applied to data assimilation problems}, ESAIM, Math. Model. Numer. Anal., 56 (2022), pp.~2021--2050, \url{https://doi.org/10.1051/m2an/2022061}.

\bibitem{BR18}
{\sc L.~Bourgeois and A.~Recoquillay}, {\em A mixed formulation of the {T}ikhonov regularization and its application to inverse {PDE} problems}, ESAIM Math. Model. Numer. Anal., 52 (2018), pp.~123--145, \url{https://doi.org/10.1051/m2an/2018008}, \url{https://doi.org/10.1051/m2an/2018008}.

\bibitem{BLP97}
{\sc J.~H. Bramble, R.~D. Lazarov, and J.~E. Pasciak}, {\em A least-squares approach based on a discrete minus one inner product for first order systems}, Math. Comp., 66 (1997), pp.~935--955, \url{https://doi.org/10.1090/S0025-5718-97-00848-X}, \url{https://doi.org/10.1090/S0025-5718-97-00848-X}.

\bibitem{BLP98}
{\sc J.~H. Bramble, R.~D. Lazarov, and J.~E. Pasciak}, {\em Least-squares for second-order elliptic problems}, vol.~152, 1998, pp.~195--210, \url{https://doi.org/10.1016/S0045-7825(97)00189-8}, \url{https://doi.org/10.1016/S0045-7825(97)00189-8}.
\newblock Symposium on Advances in Computational Mechanics, Vol. 5 (Austin, TX, 1997).

\bibitem{Brenner}
{\sc S.~C. Brenner and L.~R. Scott}, {\em The mathematical theory of finite element methods}, vol.~15 of Texts Appl. Math., New York, NY: Springer, 3rd ed.~ed., 2008, \url{https://doi.org/10.1007/978-0-387-75934-0}.

\bibitem{Brezis}
{\sc H.~Brezis}, {\em Functional analysis, {Sobolev} spaces and partial differential equations}, Universitext, New York, NY: Springer, 2011.

\bibitem{Bu13}
{\sc E.~Burman}, {\em Stabilized finite element methods for nonsymmetric, noncoercive, and ill-posed problems. {P}art {I}: {E}lliptic equations}, SIAM J. Sci. Comput., 35 (2013), pp.~A2752--A2780, \url{https://doi.org/10.1137/130916862}, \url{https://doi.org/10.1137/130916862}.

\bibitem{Bu14}
{\sc E.~Burman}, {\em Error estimates for stabilized finite element methods applied to ill-posed problems}, C. R. Math. Acad. Sci. Paris, 352 (2014), pp.~655--659, \url{https://doi.org/10.1016/j.crma.2014.06.008}, \url{https://doi.org/10.1016/j.crma.2014.06.008}.

\bibitem{Bu16}
{\sc E.~Burman}, {\em Stabilised finite element methods for ill-posed problems with conditional stability}, in Building bridges: connections and challenges in modern approaches to numerical partial differential equations, vol.~114 of Lect. Notes Comput. Sci. Eng., Springer, [Cham], 2016, pp.~93--127.

\bibitem{BDE21}
{\sc E.~Burman, G.~Delay, and A.~Ern}, {\em A hybridized high-order method for unique continuation subject to the {H}elmholtz equation}, SIAM J. Numer. Anal., 59 (2021), pp.~2368--2392, \url{https://doi.org/10.1137/20M1375619}, \url{https://doi.org/10.1137/20M1375619}.

\bibitem{BLO}
{\sc E.~Burman, M.~G. Larson, and L.~Oksanen}, {\em Primal-dual mixed finite element methods for the elliptic {C}auchy problem}, SIAM Journal on Numerical Analysis, 56 (2018), pp.~3480--3509, \url{https://doi.org/10.1137/17M1163335}, \url{https://doi.org/10.1137/17M1163335}, \url{https://arxiv.org/abs/https://doi.org/10.1137/17M1163335}.

\bibitem{BNO19}
{\sc E.~Burman, M.~Nechita, and L.~Oksanen}, {\em Unique continuation for the {H}elmholtz equation using stabilized finite element methods}, J. Math. Pures Appl. (9), 129 (2019), pp.~1--22, \url{https://doi.org/10.1016/j.matpur.2018.10.003}, \url{https://doi.org/10.1016/j.matpur.2018.10.003}.

\bibitem{BNO20}
{\sc E.~Burman, M.~Nechita, and L.~Oksanen}, {\em A stabilized finite element method for inverse problems subject to the convection-diffusion equation. {I}: diffusion-dominated regime}, Numer. Math., 144 (2020), pp.~451--477, \url{https://doi.org/10.1007/s00211-019-01087-x}, \url{https://doi.org/10.1007/s00211-019-01087-x}.

\bibitem{BNO22}
{\sc E.~Burman, M.~Nechita, and L.~Oksanen}, {\em A stabilized finite element method for inverse problems subject to the convection-diffusion equation. {II}: convection-dominated regime}, Numer. Math., 150 (2022), pp.~769--801, \url{https://doi.org/10.1007/s00211-022-01268-1}, \url{https://doi.org/10.1007/s00211-022-01268-1}.

\bibitem{burman2023optimal}
{\sc E.~Burman, M.~Nechita, and L.~Oksanen}, {\em Optimal finite element approximation of unique continuation}, 2023, \url{https://arxiv.org/abs/2311.07440}.

\bibitem{BO18}
{\sc E.~Burman and L.~Oksanen}, {\em Data assimilation for the heat equation using stabilized finite element methods}, Numer. Math., 139 (2018), pp.~505--528, \url{https://doi.org/10.1007/s00211-018-0949-3}, \url{https://doi.org/10.1007/s00211-018-0949-3}.

\bibitem{burman2023finite}
{\sc E.~Burman and L.~Oksanen}, {\em Finite element approximation of unique continuation of functions with finite dimensional trace}, 2023, \url{https://arxiv.org/abs/2305.06800}.

\bibitem{Chervova}
{\sc O.~Chervova and L.~Oksanen}, {\em Time reversal method with stabilizing boundary conditions for photoacoustic tomography}, Inverse Problems, 32 (2016), p.~125004, \url{https://doi.org/10.1088/0266-5611/32/12/125004}, \url{https://doi.org/10.1088%2F0266-5611%2F32%2F12%2F125004}.

\bibitem{CIY22}
{\sc E.~Chung, K.~Ito, and M.~Yamamoto}, {\em Least squares formulation for ill-posed inverse problems and applications}, Appl. Anal., 101 (2022), pp.~5247--5261, \url{https://doi.org/10.1080/00036811.2021.1884228}, \url{https://doi.org/10.1080/00036811.2021.1884228}.

\bibitem{nicolae2015}
{\sc N.~C{\^{\i}}ndea and A.~M{\"u}nch}, {\em Inverse problems for linear hyperbolic equations using mixed formulations}, Inverse Probl., 31 (2015), p.~38, \url{https://doi.org/10.1088/0266-5611/31/7/075001}.
\newblock Id/No 075001.

\bibitem{CM79}
{\sc P.~Colli~Franzone and E.~Magenes}, {\em On the inverse potential problem of electrocardiology}, Calcolo, 16 (1979), pp.~459--538 (1980), \url{https://doi.org/10.1007/BF02576643}, \url{https://doi.org/10.1007/BF02576643}.

\bibitem{dahmen2022squares}
{\sc W.~Dahmen, H.~Monsuur, and R.~Stevenson}, {\em Least squares solvers for ill-posed {PDE}s that are conditionally stable}, ESAIM Math. Model. Numer. Anal., 57 (2023), pp.~2227--2255, \url{https://doi.org/10.1051/m2an/2023050}, \url{https://doi.org/10.1051/m2an/2023050}.

\bibitem{DHH13}
{\sc J.~Dard\'{e}, A.~Hannukainen, and N.~Hyv\"{o}nen}, {\em An {$H_{{div}}$}-based mixed quasi-reversibility method for solving elliptic {C}auchy problems}, SIAM J. Numer. Anal., 51 (2013), pp.~2123--2148, \url{https://doi.org/10.1137/120895123}, \url{https://doi.org/10.1137/120895123}.

\bibitem{ern}
{\sc A.~Ern and J.-L. Guermond}, {\em Theory and practice of finite elements.}, vol.~159 of Appl. Math. Sci., New York, NY: Springer, 2004.

\bibitem{FM86}
{\sc R.~S. Falk and P.~B. Monk}, {\em Logarithmic convexity for discrete harmonic functions and the approximation of the {C}auchy problem for {P}oisson's equation}, Math. Comp., 47 (1986), pp.~135--149, \url{https://doi.org/10.2307/2008085}, \url{https://doi.org/10.2307/2008085}.

\bibitem{Grisvard_Ell}
{\sc P.~Grisvard}, {\em Elliptic problems in nonsmooth domains}, vol.~24 of Monogr. Stud. Math., Pitman, Boston, MA, 1985.

\bibitem{IYH91}
{\sc D.~B. Ingham, Y.~Yuan, and H.~Han}, {\em The boundary-element method for an improperly posed problem}, IMA J. Appl. Math., 47 (1991), pp.~61--79, \url{https://doi.org/10.1093/imamat/47.1.61}, \url{https://doi.org/10.1093/imamat/47.1.61}.

\bibitem{IJ15}
{\sc K.~Ito and B.~Jin}, {\em Inverse problems}, vol.~22 of Series on Applied Mathematics, World Scientific Publishing Co. Pte. Ltd., Hackensack, NJ, 2015.
\newblock Tikhonov theory and algorithms.

\bibitem{Joh04}
{\sc T.~Johansson}, {\em An iterative procedure for solving a {C}auchy problem for second order elliptic equations}, Math. Nachr., 272 (2004), pp.~46--54, \url{https://doi.org/10.1002/mana.200310188}, \url{https://doi.org/10.1002/mana.200310188}.

\bibitem{KK95}
{\sc S.~I. Kabanikhin and A.~L. Karchevsky}, {\em Optimizational method for solving the {C}auchy problem for an elliptic equation}, J. Inverse Ill-Posed Probl., 3 (1995), pp.~21--46, \url{https://doi.org/10.1515/jiip.1995.3.1.21}, \url{https://doi.org/10.1515/jiip.1995.3.1.21}.

\bibitem{LL69}
{\sc R.~Latt\`es and J.-L. Lions}, {\em The method of quasi-reversibility. {A}pplications to partial differential equations}, Modern Analytic and Computational Methods in Science and Mathematics, No. 18, American Elsevier Publishing Co., Inc., New York, 1969.
\newblock Translated from the French edition and edited by Richard Bellman.

\bibitem{RHD99}
{\sc H.-J. Reinhardt, H.~Han, and D.~N. H\`ao}, {\em Stability and regularization of a discrete approximation to the {C}auchy problem for {L}aplace's equation}, SIAM J. Numer. Anal., 36 (1999), pp.~890--905, \url{https://doi.org/10.1137/S0036142997316955}, \url{https://doi.org/10.1137/S0036142997316955}.

\bibitem{Scott-Zhang}
{\sc L.~R. Scott and S.~Zhang}, {\em Finite element interpolation of nonsmooth functions satisfying boundary conditions}, Mathematics of Computation, 54 (1990), pp.~483--493, \url{http://www.jstor.org/stable/2008497} (accessed 2023-10-08).

\bibitem{TA77}
{\sc A.~N. Tikhonov and V.~Y. Arsenin}, {\em Solutions of ill-posed problems}, Scripta Series in Mathematics, V. H. Winston \& Sons, Washington, D.C.; John Wiley \& Sons, New York-Toronto-London, 1977.
\newblock Translated from the Russian, Preface by translation editor Fritz John.

\end{thebibliography}
\end{document}